\documentclass[a4paper,reqno]{amsart}

\usepackage{amsbsy}

\usepackage{amsmath}
\usepackage{paralist}
\usepackage{graphics}
\usepackage{epsfig}
\usepackage{amsfonts}
\usepackage{amssymb}
\usepackage{hyperref,cite}

\usepackage{amsthm}
\usepackage{enumerate}
\usepackage[mathscr]{eucal}
\usepackage{eqlist}
\usepackage{amsmath}

\def\poin{{Poincar$\acute{e}$ }}

\newtheorem{theorem}{Theorem}[section]
\newtheorem{lemma}[theorem]{Lemma}

\def\ifl{\iffalse }

\numberwithin{equation}{section}

\numberwithin{equation}{section}

\theoremstyle{remark}

\begin{document}
\baselineskip=1.2\baselineskip
\title[  ]
{Markov maps and quasi-conformal deformations associated to geometrically finite Kleinian groups }

\author{Shengjin Huo}
\address{Department of Mathematics, Tiangong University, Tianjin 300387, China} \email{huoshengjin@tiangong.edu.cn}

\author{Wanting Shao}
\address{Department of Mathematics, Tiangong University, Tianjin 300387, China} \email{2810140065@qq.com}


\thanks{This work was supported by Tianjin Natural Science Foundation of China (23JCYBJC00730).}



\begin{abstract}
	It is known that the dynamical properties of Kleinian groups with parabolic elements are much more complex than those without parabolic elements. In this paper, we mainly study the quasi-conformal deformation properties of geometrically finite Kleinian groups with parabolic elements. We show that, for a Kleinian group $G$ with parabolic elements, there exists a subset $K$ of the limit set $\Lambda(G)$ such that the first return map $f_{K}$ is (eventually) expending. Furthermore, the expending property is invariant under quasiconformal deformations. As an application, we show that  the geometrically finite Kleinian groups have Ruelle's property. Finally, we use  eigenvalue algorithm  to
	calculate  the precise Hausdorff dimensions of Hecke groups, which serves as an example to support our results.
\end{abstract}

\maketitle

\section{Introductioin}

We say $G$ is a Kleinian group if it is a discrete isometry subgroup of the orientation preserving isometry group PSL$(2, \mathbb{C})$ which contains all the isometries of the hyperbolic $3$-space  $\mathbb{H}^{3}.$
Here  the hyperbolic $3$-space  $\mathbb{H}^{3} $ is  modeled as the upper half-space $\{(x_{1},x_{2},x_{3})\in \mathbb{R}^{3}:x_{3}>0\}$ in $\mathbb{R}^{3}$(It also can be modeled by the unit ball $\mathbb{B}^{3}$ in the subsequent sections). The boundary $\partial\mathbb{H}^{3}$ of $\mathbb{H}^{3}$ at infinity is the Riemann sphere $ \hat{\mathbb{C}}$ which is the one-point compactification of the complex plane $\mathbb{C}$  by adding the point $\infty$ at infinity.
The limit set $\Lambda (G)$ in the infinity boundary $\partial\mathbb{H}^{3}$ is the accumulation set  of the $G$-orbits of any point $z\in \mathbb{H}^{3}.$
A point $x\in \partial\mathbb{H}^{3}$ is called a conical limit point or radial limit point of $G$ if there is a sequence $(\gamma _{i})_{i\geq1}$ of elements $\gamma _{i}\in G$ such that for any $z\in \mathbb{H}^{3}$, there exists a hyperbolic line $L$ with endpoint $x$ such that the hyperbolic distance between $\gamma _{i}(z)$ and $L$ is bounded. Denote by $\Lambda_{c}(G)$ the set of all conical limit points of $G$.

An element $\gamma $ in $G$ is called parabolic if it fixes only one point $p$ in $\partial\mathbb{H}^{3}$  and the point $p$ is called a parabolic fixed point of $G.$
Denote by $\Lambda_{P}(G)$ the set of all parabolic points of $G.$ It is known that $\Lambda_{P}(G)$ is a subset of the limit set $\Lambda (G) .$
For each point $p\in \Lambda_{P}(G),$ the  stabilizer $G_{p} = \left\{ \gamma \in G : \gamma (p)=p \right\}$ is a parabolic subgroup of $G$. Similarly, an element $\gamma$ in $G$ is called loxodromic if it fixes two points in $\partial\mathbb{H}^{3}.$ In this paper, we suppose $G$ does not contain elements fixing points in $\mathbb{H}^{3}$, i.e. $G$ is torsion free.
It is known that, in a discrete subgroup, a parabolic element and a loxodromic element cannot share a common fixed point \cite{KS}. Thus, $G_{p}$ consists entirely of parabolic elements.
A parabolic fixed point $p$ is said to be a bounded parabolic fixed point if the quotient set $ (\Lambda \left ( G \right ) \setminus \left \{ p \right \})/G_{p} $  is compact. It is easy to see that $G_{p}$ is isomorphic to either $\mathbb{Z}$ or $\mathbb{Z} \oplus \mathbb{Z}.$ Taking $B$ to be any horoball about $p,$ if $G_{p}\cong \mathbb{Z},$ then $\partial B/G_{p}$ is a bi-infinite  cylinder which is homeomorphic to $\mathbb{S}^{1} \times \mathbb{R} \times (0,\infty).$ In this case $B/G_{p}$ is a $\mathbb{Z}$-cusp of rank 1.
If $G_{p} \cong \mathbb{Z}\oplus \mathbb{Z}$ then $\partial B/G_{p}$ is a torus which is homeomorphic to $\mathbb{S}^{1} \times \mathbb{S}^{1} \times (0,\infty),$ and in this case $\partial B/G_{p}$ is a $\mathbb{Z}\oplus \mathbb{Z}$-cusp of rank 2, see FIGURE 1.1.

\begin{figure}[htbp]
	\centering
	\raisebox{-10em}{\includegraphics[height=24.8em]{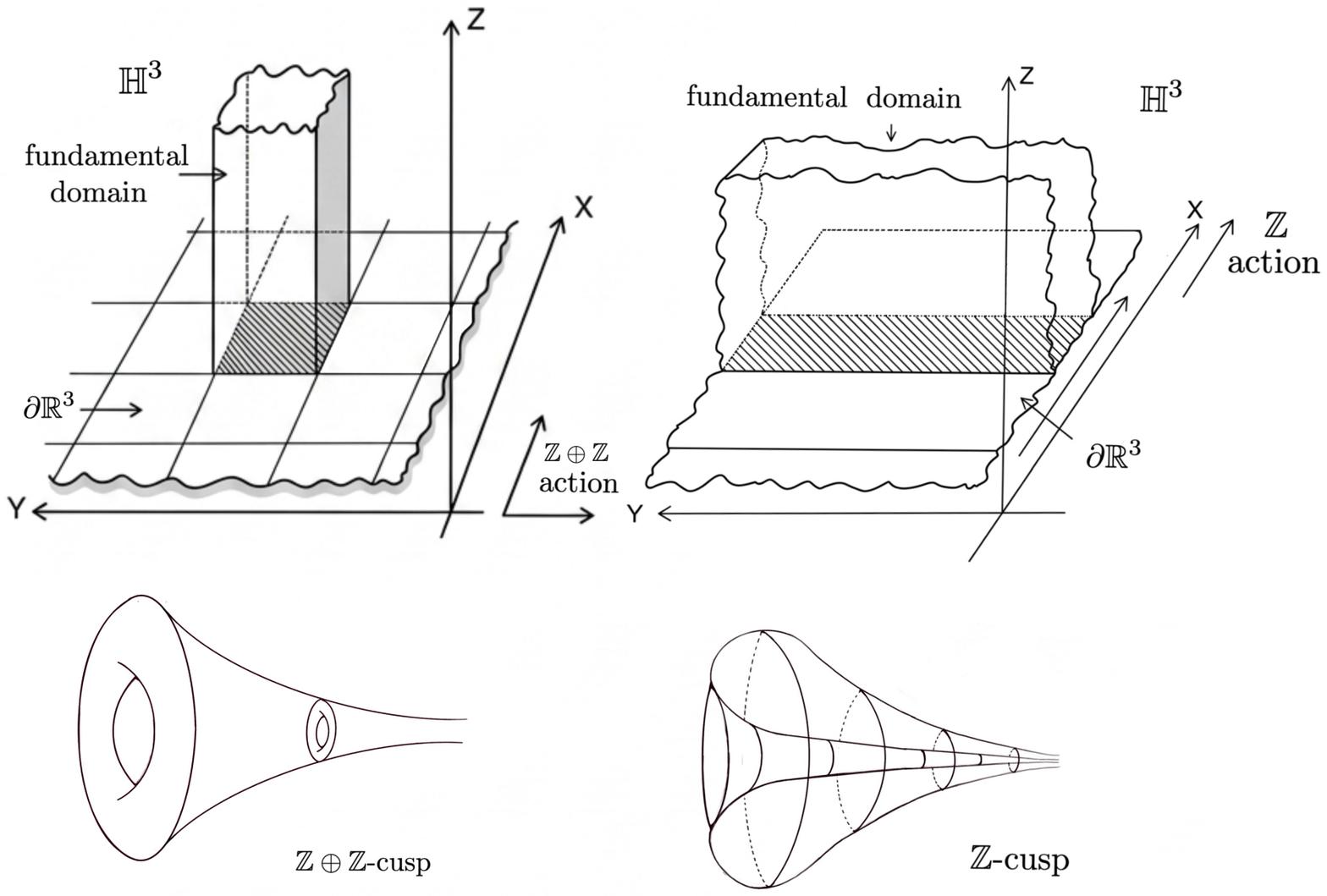}}
	\caption{1}
\end{figure}


We denote by $\text{Hull}\left (\Lambda \right )\subset \mathbb{H}^{3}$ the closed convex hull of $\Lambda \left ( G \right )  \subset \partial\mathbb{H}^{3}$, which is the smallest closed convex subset in  $\mathbb{H}^{3}$ whose accumulation set in $\partial\mathbb{H}^{3}$  is  $\Lambda \left ( G \right ) $.
The quotient $\text{Hull}\left (\Lambda \right )/G  $ is called the convex core of $G.$
A Kleinian group $G$ is convex co-compact if its convex core is compact and is co-finite if the quotient surface $\mathbb{H}^{3}/G $ has finite hyperbolic area.

The notion of geometrically finite Kleinian groups was originally introduced by Ahlfors \cite{A}, for subgroups of   PSL$(2, \mathbb{C})$ as the finiteness condition for the number of faces of a  convex fundamental polyhedron. Beardon and Maskit\cite{BM} prove that a Kleinian group $G$ of $\mathbb{H}^{3}$ is geometrically finite if and only if the limit set $\Lambda(G)$ consists entirely of conical limit points and bounded parabolic fixed points. Several equivalent definitions of geometrical finiteness have been work out, see Bowditch \cite{Bow} for more details.

%

Let $K$ be a subset of the Riemann sphere $\hat{\mathbb{C}}$. We say that a real analytic map $f$ is expanding if $|f'|_{\sigma}\geq \beta$ for any $x\in K$ and for some positive constant $\beta>1$,
where $|f'|_{\sigma}= |f'| \frac{\sigma(f)}{\sigma}$ called the conformal derivative of $f$  is measured in the spherical metric  $\sigma = 2|dx|/(1+|x|^{2})$. Similarly, we say that a  map $f$ is eventually expanding if there exists  a positive integer $N$ such that $|(f^{N})'|_{\sigma}\geq \beta>1$ for any $x\in K$.

Let $\Gamma$ be a Fuchsian group, i.e. a discrete isometry subgroup of the orientation preserving isometry group PSL$(2, \mathbb{R}).$  Let $f_{\Gamma}$ be a Markov map of $\Gamma$ (see $\S$2 for definition).  In 1979, Bowen and Series  \cite{BS} showed that  if $\Gamma$ has no parabolic elements, then the Markov map $f_{\Gamma}$ is eventually expanding. If $\Gamma$ has parabolic elements, then there is a subset $K$ of the unit circle $ \mathbb{S} ^{1} $ such that the first return map $f_{K}$ on $K$ is eventually expanding, where
$$f_{K}:K\rightarrow K, f_{K}(x)=f_{\Gamma}^{m(x)}(x), m(x)=\inf\{m>0: f_{\Gamma}^{m(x)}(x)\in K\}.$$
In 1997, Anderson and Rocha \cite{AR} extended Bowen and Series' results to co-compact Kleinian groups without parabolic elements. In this paper, we extend the results to  cases of geometrically finite Kleinian groups with parabolic elements as follows.

\begin{theorem}
	Let $G$ be a geometrically finite Kleinian group with parabolic elements and $f_{G}$ be the associated Markov map. Suppose $G$ has an even-cornered fundamental
	domain $\mathcal{F}$ and let $P$ denote the finite set of all parabolic vertices of $\mathcal{F},$ and define $\mathcal{P}:=\bigcup_{n=0}^{\infty} f_{G}^{-n}(P),$ which is the set of all parabolic fixed points of $G.$ Then for each $p\in P,$
	we can choose a sufficiently small neighborhood $U_p$ of $p$ such that the first return map $f_K$ induced by $f_{G}$ on $K$
	is eventually expanding, where
	$$K:=\Lambda(G)\setminus
	\left(
	\bigcup_{p\in P}U_p\;\cup\;\mathcal{P}
	\right)
	\subset \Lambda_c(G).$$
\end{theorem}

In the theory of Kleinian groups, one of the main themes is to study the dimensional properties of the limit sets of the associated Kleinian groups under quasi-conformal deformations.
Ahlfors' finiteness Theorem implies that the limit set of a geometrically finite Kleinian group $G$ is either the whole Riemann sphere $\partial\mathbb{H}^{3}=\hat{\mathbb{C}}$  or a closed subset of $\hat{\mathbb{C}}$. This partitions the  Riemann sphere $\hat{\mathbb{C}}$ into two disjoint sets: the limit set and the regular set. Here the regular set $\Omega(G)$ is the largest open set in $\hat{\mathbb{C}}$ on which $G$ acts properly discontinuously. In this paper, we mainly focus on the case of the geometrically finite Kleinian groups with $\Omega(G)\neq\emptyset$.

Suppose $\mu$ is a bounded measurable function on $\Omega(G)$ with $\|\mu(z)\|_{\infty}<1.$
We say $\mu$ is a $G$-invariant Beltrami coefficient  of $G$ if $\mu(z)=\mu(\gamma(z))\overline{\gamma'(z)}/\gamma'(z)$ for every $\gamma\in G $ and a.e. $z\in \Omega \left ( G \right ) .$
We denote by $\mathrm {M} \left ( G \right )$ the set of all  $G$-invariant Beltrami coefficients. For each $\mu\in \mathrm {M}(G)$, there is exactly unique solution of the Beltrami equation $g_{\bar{z}}=\mu g_{z}$ fixing $0,$ $1$, and $\infty,$ denoted by $g^{\mu}.$ Such a quasiconformal mapping $g^{\mu}$ is called being compatible with the group $G.$ We call $G_{\mu}=g^{\mu}\circ G \circ (g^{\mu})^{-1}$ a quasiconformal deformation group of  Kleinian group $G$. Note that the condition $\mu(z)=\mu(\gamma(z))\overline{\gamma'(z)}/\gamma'(z)$ for every $\gamma\in G $  is exactly the condition needed to assert that $g^{\mu}\circ \gamma \circ (g^{\mu})^{-1}$ is again an element of  PSL$(2, \mathbb{C}),$ hence the deformation group $G_{\mu}$ is also a subgroup of PSL$(2, \mathbb{C})$.  Two elements $\mu$ and $\tau$ in $\mathrm {M} \left ( G \right )$ are called equivalent, denoted $\mu\sim \tau$, if $g^{\mu}$ and  $g^{\tau}$  are equal on the limit set $\Lambda(G).$ Denote the equivalence class of $\mu\in \mathrm {M} \left ( G \right )$ by $[\mu]$. The quotient space of $\mathrm {M} \left ( G \right )$
by this equivalence relation is the deformation space $\mathcal{T}(G)$ which is also called the Teichm\"{u}ller space of $G$.

By Theorem 1.1, we know that the map $f_{K}$ is (eventually) expanding. In the following, we will shown that, under quasiconformal conjugation-deformations, the map $f_{K}^{\mu } =g^{\mu}\circ f_{K} \circ (g^{\mu})^{-1} $ is also (eventually) expanding.

\begin{theorem}The map $f_{K}^{\mu } $ is (eventually) expanding.
	
\end{theorem}

It is well known that whenever $\{\mu_{t}\}$ is an analytic family of Beltrami coefficients with respect to the parameter $t$, the family of quasiconformal maps $\{g^{\mu_{t}}\}$ is also  analytic with respect to $t.$ Hence, $G_{\mu_{t}}$ is an analytic family of quasiconformal deformations of $G.$ An element $\mu$ in $\mathrm {M}(G)$ also induces a discrete faithful representation $\rho ^{\mu }  : G \longrightarrow PSL_2(\mathbb{C}).$
The group representation $[\mu_{t}] \mapsto \rho^{\mu_{t}}$ is also real analytic on the deformation space $\mathcal{T}(G).$ We say that a discrete group $G$ has Ruelle's property if whenever $G_{\mu_{t}}$ is an analytic family of quasiconformal deformations of $G,$ then dim$_{H}\Lambda (G_{\mu_{t}})$ is an analytic function of $t,$ where dim$_{H}(\Lambda)$ denotes the Hausdorff dimension of  the set.

There is a long history to study Ruelle's property of Fuchsian groups, we do not recall here.  The readers can see \cite{HZ} and the references therein for more details.  In the theory of complex dynamic system, there are also many researchers to study Ruelle's property of analytic functions, see \cite{MA}.  For Kleinian groups, Anderson and Rocha \cite{AR} proved  that the convex cocompact Kleinian groups have Ruelle's property. Bridgeman et al. \cite{MR} showed that the Hausdorff dimension of the limit set varies analytically over analytic families of convex cocompact representations into any rank-one semisimple Lie group. Dey and Hee Oh \cite{HD} proved that, for any $\theta$-Anosov Zariski-dense subgroup of a higher-rank connected semisimple real algebraic group, the Hausdorff dimension of the limit set also varies analytically. In the presence of parabolic elements, the first author \cite{HZ} proved  that every finitely generated Fuchsian group has Ruelle's property. Bray et al. \cite{HRL} constructed a mapping class group-invariant pressure metric on the space $\mathrm{QF}(S)$ of quasiconformal deformations of a cofinite-area Fuchsian group uniformizing $S.$ Using this method, they also showed that for cofinite Fuchsian groups, the Hausdorff dimension of the limit set varies analytically. In this paper, we continue to study the Ruelle's property of Kleinian groups.

Recall that, for a Kleinian group G, a fundamental domain $\mathcal{F} $ for the action of $G$ on $\mathbb{H}^{3}$ is an open subset of $\mathbb{H}^{3}$ such that the various images of $\overline{\mathcal{F} }$ by the elements of $G$ cover the whole space and they are pairwise disjoint except for points on the boundary of $\overline{\mathcal{F} }$. A fundamental domain $\mathcal{F} $ is said to be an even-cornered fundamental polyhedron if $\cup_{\gamma \in G}^{}\gamma(\partial \mathcal{F} )$ is a union of geodesic planes.

By Theorem 1.1 and Theorem 1.2, we can extend Anderson and Rocha's results \cite{AR} to the case of geometrically finite Kleinian groups.

\begin{theorem}\label{main}
	Let $G$ be a geometrically finite Kleinian group which has an even-cornered fundamental polyhedron.  Then $G$ has Ruelle's property.
\end{theorem}

The rest of this paper is organized as follows. In section 2, we review some preliminary knowledge on Kleinian groups. In Section 3, we give  a number of results about Markov maps of the geometrically finite Kleinian groups with parabolic elements and construct a subset $K$ of $\Lambda(G)$ such that the first return map $f_{K}$ is (eventually) expanding. In Section 4, we give the proof of  Theorem 1.2 and  Theorem 1.3. Finally, in Section 5, we use  eigenvalue algorithm (see \cite{Mc2}) to
calculate the precise Hausdorff dimensions of Hecke groups, which serves as an example to support our results.

\section{Preliminaries}

\subsection{Kleinian group} Let $G$ be a Kleinian group. The \poin series of $G$ is defined  by

\begin{align}P_{s}(G,x) = \left\{\begin{matrix}\sum_{\gamma \in G}^{}e^{-sd_{\mathbb{H} } (x,\gamma x)},\quad \text{if } x\in \mathbb{H}^{3}, \\
		\\
		\textstyle \sum_{\gamma \in G} |\gamma'(x)|^{s}_{\sigma},
		\quad  \text{if } x\in \Omega(G). \end{matrix}\right.
\end{align}
where $|\gamma'|_{\sigma }$ is measured in the spherical metric $\sigma = \left | \mathrm{d}x \right | /\left ( 1+x^{2}  \right ) $ and $d_{\mathbb{H}}$ denotes the hyperbolic metric on $\mathbb{H}^3 .$ The \poin exponent (or critical exponent) is given by $$\delta(G) = \inf\{ s\ge 0 : P_{s}(G,x) < \infty \},$$ which is independent of the choice of $x$ in $\mathbb{H}^{3}\bigcup \Omega(G).$

If the Kleinian group $G$ has a fundamental domain with finitely many faces, then it is finitely generated. Furthermore, if a Fuchsian group or a Kleinian group is geometrically finite, then all of its fundamental domains  have finitely many faces, and in this case the group  is finitely generated , see \cite{AM}. For the case of Fuchsian groups, the converse is also true, i.e. every finitely generated Fuchsian group is geometrically finite. However this is false for Kleinian groups, see \cite{LG}.

Let $G$ be a geometrically finite Kleinian group and $\mathcal{F} $ be one of its fundamental domains.  Suppose $s$ is a face of $\mathcal{F} $, then there exists another face $s'$ of $F $ and an element $\gamma _{s}\in G$ with $\gamma _{s}(s)=s'.$ It is easy to see that  that $\gamma _{s'}= \gamma _{s}^{-1}.$
We call $\gamma _{s}$ is a face pairing transformation of $\mathcal{F}$,  and denoted by $G_{\mathcal{F}}$ the set of all face pairing transformations of $\mathcal{F}.$

The Poincare polyhedron Theorem (see \cite{FH})  says that a geometrically finite Kleinian group is generated by its face pairing transformations.
Hence, we can get a presentation $\langle G_{\mathcal{F} }\mid R_{\mathcal{F} }\rangle$ of the group $G,$ where $G_{\mathcal{F}}$ is a symmetric system of generators and $R_{\mathcal{F} }$ is the set of relations from each vertex and face cycles, see\cite{BM}.

\subsection{Markov map}

Just as Bourdon did in  \cite{MB},  we use the elements in $G_{\mathcal{F}}$ to label the faces of $G$ as follows.
For each $\gamma\in G_{\mathcal{F}}$, for the sake of simplicity in notation with no risk of confusion, we also  use $\gamma$ to denote the inside face of $\mathcal{F}$ associated to the element $\gamma\in G_{\mathcal{F}}$ pairing it to some other face, and denote $\bar{\gamma}$ the outside face of $\mathcal{F}$.
It is easy to see that $\bar{\gamma}$ is corresponding to $\gamma^{-1}\in G_{\mathcal{F}}.$
Let $H(\gamma)$ be the part of hyperbolic half-space in $\mathbb{H}^{3}$ with $\mathcal{F}$ in its exterior and the face $\gamma^{-1} $ of $\mathcal{F}$ on its geodesic boundary.
Then $\{ H(\gamma_{i}) : \gamma_{i}\in G_{\mathcal{F}}\}\cup \mathcal{F}$ partitions $\mathbb{H}^{3}$ into several parts.
Hence, we can use the associated faces-pairings to define a map $f$ that sends the point $x$ in $H(\gamma)$ to $\gamma^{-1}(x)$ and sends the points $x\in \mathcal{F}$  to itself.
We must point out that that this map is not well defined since  some of the half-spaces $H( \gamma _{i})$ can have a non-empty intersection for different $ \gamma _{i}\in G_{\mathcal{F}}$.
To overcome this confusion, we can  choose an order for the elements in the generator set $G_{\mathcal{F}}\cup \{Id\}$ as  $\{ \gamma _{0}=Id,  \gamma _{1},  \gamma _{2}, \cdot\cdot\cdot,  \gamma _{n}\},$  where $Id$ denotes the identity map of $\mathbb{H}^{3}.$
Then for a point $x\in \mathbb{H}^{3}$ belonging to $H( \gamma _{i_{1}})\cap H( \gamma _{i_{2}})\cap\dots\cap H( \gamma _{i_{k}}),$ we define $f(x)=  \gamma _{k}^{-1}x,$ where $ \gamma _{k}$ is the element in the set
$\{ \gamma _{i_{1}},  \gamma _{i_{2}}, \cdot\cdot\cdot,  \gamma _{i_{k}}\}$  with the least foot tag. Define, see FIGURE 2.1,
$$B( \gamma _{i}) = \{ x\in H( \gamma _{i}) : f(x)=  \gamma _{i}^{-1}x \},$$
$$I( \gamma _{i}) = \widehat{\mathbb{C}}\cap \overline{B( \gamma _{i})},$$
$$B( \gamma _{0},\dots, \gamma _{n}) = \bigcap_{r=0}^{n} f^{-r}B( \gamma _{r}),$$
$$I( \gamma _{0},\dots, \gamma _{n}) =\widehat{\mathbb{C}}\cap \overline{B( \gamma _{0},\dots, \gamma _{n})} = \bigcap_{r=0}^{n} f^{-r}I(\gamma_{r}).$$
Observe that
$$f^{n}\mid_{B(\gamma_{0}\gamma_{1}\dots \gamma_{n-1})}= \gamma_{n-1}^{-1}\dots \gamma_{0}^{-1},$$ then it is easy to see that
$$B(\gamma_{0}\dots \gamma_{n}) = (\gamma_{0}\dots \gamma_{n-1}B(\gamma_{n})) \cap B(\gamma_{0}\dots \gamma_{n-1}).$$
By Theorem 2.2 in \cite{R}, we know that the diameters $\operatorname{diam} I(\gamma_{0},\dots,\gamma_{n})$ tend $ 0$ as $n\to 0$, which implies that $I(\gamma_{0},\dots,\gamma_{n},\dots)$ contains only a single point. Here and the following, $\operatorname{diam}(\cdot)$ always denote the Euclidean diameters of the associated sets.

\begin{figure}[htbp]
	\centering
	\raisebox{-10em}{\includegraphics[height=24.8em]{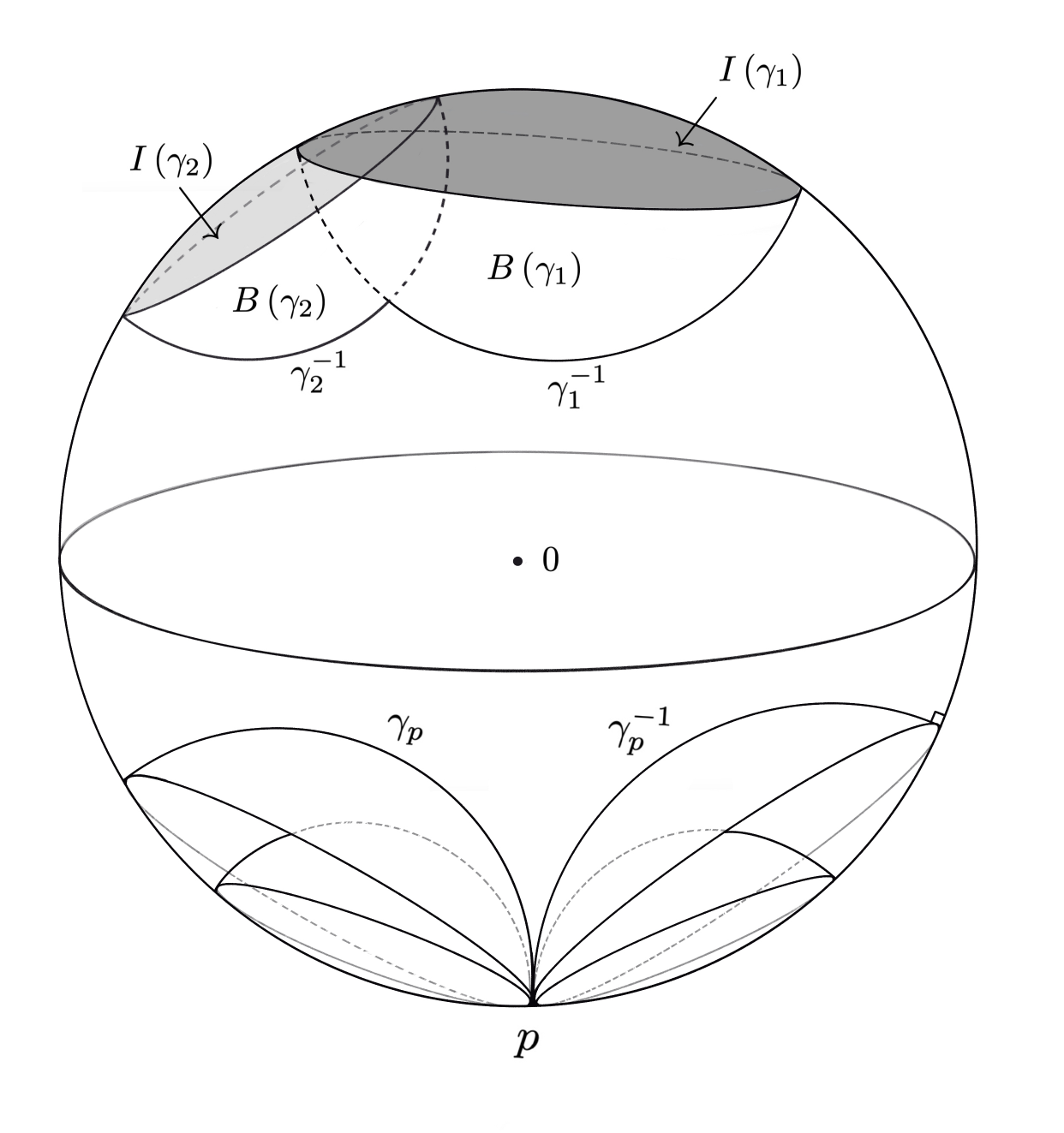}}
	\caption{1}
\end{figure}

For each word $\gamma_{i_{0}}\gamma_{i_{1}}\dots \gamma_{i_{n}}$ written by the generators in $G_{\mathcal{F}}$,  a family of planes $P(\gamma_{i_{0}}), P(\gamma_{i_{0}}\gamma_{i_{1}}),\dots,P(\gamma_{i_{0}}\dots \gamma_{i_{n}})$  can be defined as: $P(\gamma_{i_{0}})$ is the plane containing the face $\bar{\mathcal{F}}\cap \gamma_{i_{0}}\bar{\mathcal{F}}$ and, for $k=1,2,\dots,n,$ $P(\gamma_{i_{0}}\dots \gamma_{i_{n}})$ is the plane containing $\gamma_{i_{0}}\dots \gamma_{i_{k-1}}\bar{\mathcal{F}}\cap \gamma_{i_{0}}\dots \gamma_{i_{k}}\bar{\mathcal{F}}.$
All these planes give a decomposition of $\mathbb{H}^{3}$  and we can  use this decomposition to construct a Markov map on $\widehat{\mathbb{C}}$.

Consider the set $\mathcal{B}$ of all planes in $G(\partial \mathcal{F})$ that intersect $\mathcal{F}$ either on  a face or at a vertex of $F$. The intersection of each plane in $\mathbb{H}^{3}$ with $\widehat{\mathbb{C}}$ is a circle or a line, and the collection of all these regions whose boundaries are unions of arcs belonging to such circles covers $\widehat{\mathbb{C}}.$
These regions can intersect on the boundaries, however their interiors are disjoint.
Let $\left \{ c_{1},c_{2}...,c_{s} \right \} $ ($\left \{ c_{1},c_{2}...,c_{n},...  \right \} $) be a finite (countable) collection of (not necessarily disjoint) Jordan curves in $\mathbb{S}^{2},$ such that $\widehat{\mathbb{C}}- \left (c_{1}\cup c_{2}\cup ..   \right ) $ has finite or countable many components.
Let $R_{1} ,...,R_{s} $ or $R_{1} ,...,R_{n} ,...$ be the closures of the components which contain points of $\Lambda(G) $ in their interiors, and let $R_{0}$ be the union of the closures of the remaining components. So $R_{0}$ contains the closure of a fundamental domain for the action of $G$ on $\Omega \left ( G \right ).$  Note that $R_{i}$ is connected for all $i\ge 1$ . This gives us either a finite or countable partition $\mathfrak{R}= \left \{ R_{i}  \right \} _{i=0}^{\infty }  $ of $\widehat{\mathbb{C}}$ minus a set of Lebesgue measure zero. Furthermore, suppose the set of faces-pairings transformations $G_{\mathcal{F}}$ of $G$ contains some parabolic elements.
Then the partitions at a parabolic fixed point $p$ of $G$ are derived from the projection of all planes intersecting at the point $p$ within the set $G(\partial \mathcal{F})$ onto the ideal boundary.
It is evident that under the action of the stabilizer subgroup $G_{p},$ these planes with $p$ as a common vertex form the collection of images of the faces of the fundamental polygon.
Hence the parabolic fixed point $p$ corresponds to a countable number of partitions $\mathfrak{R}.$

It is not difficult to see that, for the partitions $\mathfrak{R},$ the map $f$ defined above  can be extended $C^{2}$ smoothly to the infinity boundary  $\partial \mathbb{H}^{3}=\widehat{\mathbb{C}}$  of $\mathbb{H}^{3}$. We call such $f$ a Markov map of $G$, denoted by $f_{G}$.  Rocha \cite{R} proved that  $f_{G}$  possesses the Markov property with respect to a countable partition $\mathfrak{R} = \{ R_{i} \}_{i=0}^{\infty}$ of $\widehat{\mathbb{C}}$.  The piecewise smooth map $f_{G}$ satisfies the following conditions:

$\left( A \right)$
For each $i\ge 1,$ there exists $\gamma _{i}  \in G$ so that $f_{G}\mid _{R_{i} } = \gamma _{i} \mid _{R_{i} }$,  and for $i=0$,  $f_{G}\mid _{R_{0} } $ is equal to the identity.

$\left( B \right)$ If $f_{G}(R_{i}) \cap R_{j} \ne \emptyset,$ then $f_{G}(R_{i}) \supset R_{j}.$

$\left( C \right)$ The map $f_{G}$ satisfies a transitivity condition: the image of any element $R_{i} \in \mathfrak{R} $ under $f_{G}$ is a union of elements in $ \mathfrak{R}.$

The idea of a Markov map is  to replace the action of the group $G$ on $\mathbb{S}^{2}$ by the single map $f_{G}.$ In addition, to avoid these intersections cause ambiguity, we assume that $\partial R_{i} \cap \Lambda _{G} $ is finite for all $i.$

A finite sequence $\left( r_{0}, r_{1}, \dots,  r_{n} \right)$ is called  admissible if $f_{G}\left( R_{r_{k}} \right) \supseteq R_{r_{k+1}}$ for \\ $0\le k\le n-1,$ so that $$f_{G}^{n}\left( R_{r_{0}} \right) \supseteq f_{G}^{n-1}\left( R_{r_{1}} \right) \supseteq \dots \supseteq R_{r_{n}}.$$
Following the standard procedures of symbolic dynamics, we  introduce a subshift of finite type on the space $\textstyle \sum_{f}$ of all admissible sequences
by $$\tau : {\textstyle \sum_{f}} \to {\textstyle \sum_{f}}, ~~\tau \left( r_{k} \right)_{k=0}^{\infty} = \left( r_{k+1} \right)_{k=0}^{\infty},$$
where $\textstyle \sum_{f}$ is defined as $${\textstyle \sum_{f}^{}} = \left\{ \left( r_{k} \right)_{k=0}^{\infty} : f_{G}\left( R_{r_{k}} \right) \supseteq R_{r_{k+1}},\; k=0,1,\dots \right\},$$ for more details see \cite{WM}.

In the following, we will use the Markov map $f_{G}$ to get a coding for the points in the limit set $\Lambda(G)$ of $G.$
Using admissible sequences, we consider the following set
$$\Omega = \left\{ \left( \gamma _{i} \right)_{i=0}^{\infty} : B\left( \gamma _{0}\gamma _{1}\dots \gamma _{i} \right) \ne \emptyset,\; \gamma _{i} \in G_{\mathcal{F}},\; i=0,1,\dots \right\}.$$
The Markov property of $f_{G}$ guarantees that the map $$\psi : {\textstyle \sum_{f_{G}}} \to \Omega,\:\: \psi\left( (r_{k})_{k=0}^{\infty} \right) = \left( \psi(r_{k}) \right)_{k=0}^{\infty}$$ is well defined,
where for each $r_{k}$,  $\psi(r_{k}) = \gamma_{k}\in G $ if $f_{G}\mid _{R_{r_{k} } }= \gamma_{k} .$

This map $\psi$ allows us to define a map from the shift space ${\textstyle \sum_{f}}$ to the limit set $\Lambda(G)$ of $G.$
Using the fact that
$$I\left( \gamma_{0}\gamma_{1}\dots \gamma_{n} \right) \supseteq I\left( \gamma_{0}\gamma_{1}\dots \gamma_{n+1} \right)$$
and $$\operatorname{diam} I\left( \gamma_{0}\gamma_{1}\dots \gamma_{n} \right) \to 0 ~~\text{ as}~~ n\to\infty,$$
we have that if $\bar{r} = (r_{0},r_{1},\dots) \in {\textstyle \sum_{f}},$ then
$$\pi(\bar{r}) = \bigcap_{k=0}^{\infty} I\left( \psi(r_{0})\psi(r_{1})\dots\psi(r_{k}) \right) $$ is a single point in $I\left( \psi(r_{0}) \right) \cap \Lambda(G).$

Observe that points on the boundary of each $R_{k}$ can have more than one $f$-expansion, and therefore the map $\pi : {\textstyle \sum_{f}} \to \Lambda(G)$ possibly fails to be a bijection precisely at these points. By Theorem 2.2 in \cite{R}, the map $\pi$ is a bijection except in a Lebesgue null set ${\textstyle \sum^{'}} \subset {\textstyle \sum_{f}}.$
Therefore, from the perspective of Markov partition, we can obtain the following definition.

If $(r_{0}, r_{1}, \dots,  r_{n})$ is admissible, define
$$R(r_{0}\dots r_{n}) = \bigcap_{i=0}^{n} f^{-i}\left( R_{r_{i}} \right).$$
As noted in the preceding text, if $(r_{0}, r_{1}, \dots, r_{n})$ is admissible, then the restriction of $f_{G}^{n}$ to the interior of $R(r_{0}\dots r_{n})$ is equal to the restriction of $\psi \left ( r_{n-1}  \right )^{-1}    ...\psi \left ( r_{0}  \right )^{-1} .$

\subsection{Iterated function system }  Let us recall the definition of the iterated function systems (IFS), see \cite{MU}.

Let $(X,\rho)$ be a non-empty compact metric space and let $I$ be a countable index set with at least two elements.  Let
$$S = \{ \phi_i : X \to X,\; i\in I \}$$
be a collection of injective contractions from $X$ to $X$ for which there exists $0<s<1$ such that
$$\rho(\phi_i(x),\phi_i(y)) \le s\,\rho(x,y),\qquad i\in I,\; x,y\in X.$$
Then $S$ is called an iterated function system.

Let $I_{n}$ denote the space of the words of length $n$ and  $I_{\infty}$ denote the space of infinite sequences of symbols in $I.$  Let $I^{*}= \bigcup_{n\ge 1} I_{n}.$ For $$\bar{r} = \left ( r_{1}, r_{2}, ..., r_{n}  \right )  \in I_{n},\; n\ge 1,$$
set
$$\phi _{\bar{r} } = \phi _{r_{1} }\circ  \phi _{r_{2} }\circ ...\circ  \phi _{r_{n}  }.$$
If $\bar{r}  \in I^{\ast }\cup I_{\infty } $ and $n$ does not exceed the length of $\bar{r} ,$ we denote by $\bar{r}  \mid _{n} $ the word $\left ( r_{1}, r_{2},..., r_{n}  \right ).$
For $\bar{r} \in I_{\infty } ,$ the set $$\pi(\bar{r} ) = \bigcap_{n=1}^{\infty} \phi_{\bar{r}  \mid n}(X)$$
is a singleton and therefore we can define a map $\pi : I_{\infty} \to X.$

The set
$$J = \pi(I_{\infty}) = \bigcup_{\bar{r} \in I_{\infty}} \bigcap_{n=1}^{\infty} \phi_{\bar{r} \mid n}(X)$$
is called the limit set associated to the system $S = \{ \phi_i : X\to X,\; i\in I \}.$

Let $\tau : I_{\infty} \to I_{\infty}$ be the left-shift map on $I_{\infty},$ that is, $\tau \left ( \bar{r}  \right )= \left ( r_{2}r_{3}... \right )  .$ Since $\phi_i(\pi(\bar{r})) = \pi(i\bar{r} )$ for every $i\in I,$ we get
$$\pi\circ\tau(\bar{r} ) = \phi_{r_1}^{-1} \circ \pi(\bar{r} )$$
and
$$J = \bigcup_{i\in I} \phi_i(J).$$
Notice that if $I$ is finite, then $J$ is compact.

For every $\iota \ge 0,$ we define
$$\psi(\iota ) = \sum_{i\in I } \|\phi_i'\|^{\iota } \le \infty,$$
where the norm $\|\cdot\|$ is the supremum norm taken over $X.$ For $n\ge 1,$ let
$$\psi_n(\iota ) = \sum_{\bar{r} \in I_{n} } \|\phi_{\bar{r} }'\|^{\iota }.$$
By \cite{MU}, we know that
$$\psi_n(\iota ) < \infty \;\Leftrightarrow\; \psi(\iota )=\psi_1(\iota ) < \infty.$$
Let $\theta = \inf\{ \iota  : \psi(\iota ) < \infty \}.$ For $n\ge 1,$ the function $\log(\psi_n)$ is convex on $(\theta,+\infty),$ and for these values of $\theta,$
$$P(\iota ) = \lim_{n\to\infty} \frac{1}{n} \log \psi_n(\iota )$$
always exists and is finite if and only if $\psi(\iota ) < \infty.$ The function $P(\iota )$ is called the topological pressure function. By  Lemma 3.2 in \cite{MU}, we know that $P(\iota )$ is strictly decreasing in the variable $\iota $ on the interval $(\theta,+\infty).$ The IFS is regular if and only if $P(\theta) = \infty,$ which is equivalent to $\psi(\theta) = \infty.$

\section{Properties of Markov maps}
In order to avoid discussing the infinity point $\infty$ separately,  in this section we will use the hyperbolic unit ball  model $\mathbb{B}^{3}$ instead of the hyperbolic plane  model $\mathbb{H}^{3}$. In this condition, we replace the infinity boundary $\hat{\mathbb{C}}$ by the unit sphere $\mathbb{S}^{2}.$


Let $G$ be a geometrically finite Kleinian group acting on $\mathbb{B}^{3}$ with parabolic elements and $\mathcal{F}$ be one of its fundamental domains.
Since the set of faces-pairings transformations $G_{\mathcal{F}}$ contains parabolic elements. Then the Markov partitions $\mathfrak{R}$ of $G$ is countable and the quotient manifold $\mathbb{B}^{3}/G$ contains cusps.

Let $P$ denote the finite set of all parabolic vertices of an even-cornered fundamental domain for $G.$ Then each periodic orbit in $P$ corresponds to a cusp of the quotient manifold $\mathbb{B}^{3}/G.$
Each  $p\in P$ is also corresponding to a periodic point of the Markov map $f_{G}$  with $f^{n}_{G}(p)=p$ and $(f^{n}_{G})'(p)=1,$ or a point in the periodic orbit of some parabolic periodic point. Hence the expanding properties of $f_{G}$ fails.
To avoid this situation happening, we will choose a sufficiently small neighborhood $U_{p}$ (the radius of it will be given later) for each $p\in P$ such that the Markov map $f_{G}$  is expanding when it is restricted on a subset $K$ of $\Lambda(G).$
Here and the following, for the sake of symbol simplicity, we will replace $f_{G}$ by $f.$   Denote by $\mathcal{P}=\bigcup_{n= 0}^{\infty } f^{(-n)} (P)$ the set of all parabolic fixed point of $G.$

\begin{theorem}
	Let $G$ be a geometrically finite Kleinian group with parabolic elements and $f$ be the associated Markov map.  Then for every quasiconformal map $g\colon \mathbb{S}^{2}\to\mathbb{S}^{2}$ being compatible
	with $G,$ there exists a neighborhood $U_p$ , for each $p\in P$, such that the set \begin{align}
		K=\Lambda(G)\setminus
		\left(
		\mathcal{P}\cup\bigcup_{p\in P}U_p
		\right)
	\end{align}
	satisfies $\dim_H g(K)=\dim_H g\bigl(\Lambda(G)\bigr),$
	where $P$ denotes the finite set of parabolic vertices of an even-cornered fundamental domain for $G,$ $\mathcal{P} =\bigcup_{n=0}^{\infty} f^{-n}(P)$ denotes the set of all parabolic fixed points of $G, $ and $\dim_H$ denotes the Hausdorff dimension.
\end{theorem}

To prove this theorem, we  need the following well known Mass distribution principle, see \cite{KF}.

\begin{lemma}(Mass distribution principle)
	Suppose $E$ is a compact set and $s\ge 0.$ If there is a mass distribution $\mu $ on $E$ and a constant $C>0$ such that for any $r > 0$,
	$\mu \left ( V \right )\le C \text{diam}( V  )^{s}   $ for all closed sets $V$ with $\text{diam}( V  )\le r ,$ then  dim$_{H}\left ( E \right ) \ge s$.
\end{lemma}

A Borel measure $\mu_{s}$ is an $s$-conformal measure of a Kleinian group $G$ if $\mu$ is a non-trivial finite measure supported on $\Lambda(G)$ satisfying
$$\mu_{s} (\gamma (A))=\int_{A}|\gamma '|^{s} d\mu$$
for every Borel subset $A$ of $\mathbb{B}^{3}$ and every $\gamma \in G,$ where $|\gamma '|$ is the operator norm of the derivative of $\gamma.$

Patterson constructed a class of $\delta$-conformal measures of Fuchsian groups , where $\delta$ denotes the Poincare convergence exponent. Sullivan proved that Patterson's construction works also for Kleinian groups. Today the conformal measures given by Patterson's construction \cite{Pa} are known as Pattersson-Sullivan measure.

For $x\in \Lambda \left ( G \right )  $ and $t> 0,$ let $x_{t} \in \mathbb{B}^{3}  $ be the unique point on the geodesic ray joining $O$ to $x$ which is the point at hyperbolic distance $t$ from $O.$ Write $S\left ( x,t \right )\subset \mathbb{S}^{2}   $ to denote the shadow at the infinity of the 2-dimensional hyperbolic plane $H_{t}$ passing through $x_{t}$ with normal direction  along the geodesic ray from $O$ to $x.$ Denote this projection map by $\pi,$ see FIGURE 3.1.

It is easy to see that $S\left ( x, t \right )=\pi(H_{t})$ is an Euclidean ball centred at $x$ with radius uniformly comparable to $e^{-t} .$
\begin{figure}[htbp]
	\centering
	\raisebox{-10em}{\includegraphics[height=24.8em]{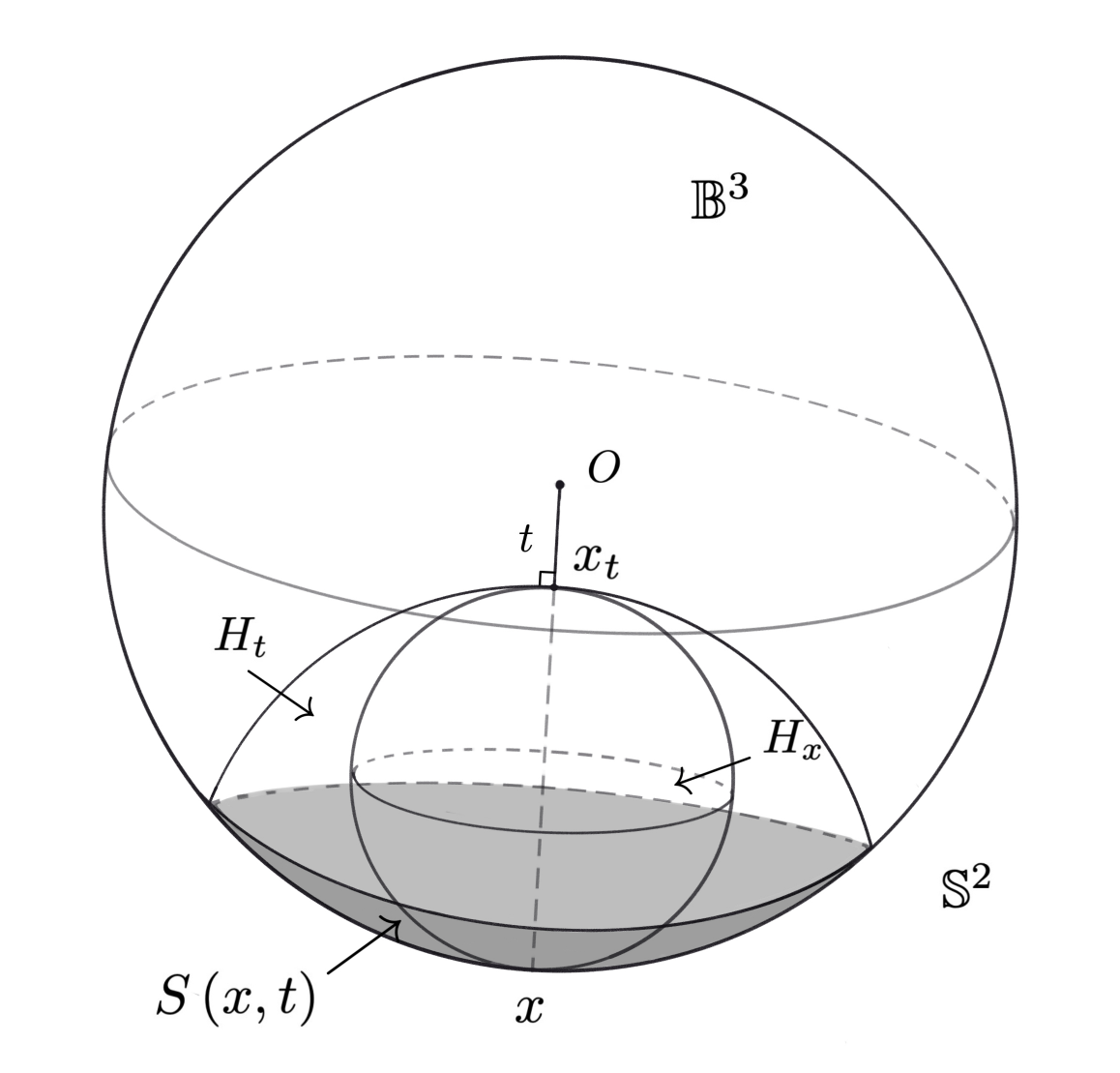}}
	\caption{1}
\end{figure}

For each point $p$ in  $P,$ let $k\left ( p \right ) $  be the maximal rank of a free Abelian subgroup of the stabiliser of $p$ in $G$. It is known that $\delta > k\left ( p \right )/2 $ for all $p\in P.$ We will use the famous global measure formula lemma as follow.

\begin{lemma}(The global measure formula,\cite{BO}\cite{VA}) For all $x\in \Lambda \left ( G \right )  $ and all positive $t> 0$, the quantity
	\begin{align}
		\mu _{\delta }  \left ( S\left ( x,t \right )  \right ) exp\left ( -t\delta -\rho \left ( x,t \right )\left ( \delta -k\left ( x,t \right ) \right )   \right )
	\end{align}
	is bounded from above and below. Here $k\left ( x,t \right ) = k\left ( p \right ) $ if $x_{t}\in H_{p}  $ for some $p$ and $0$ otherwise, and
	$$\rho \left ( x,t \right ) = \inf \left \{ d_{\mathbb{B} }\left ( x_{t},y  \right ) :y\notin H_{p}    \right \} ,$$
	if $x_{t}\in H_{p}  $ for some $p$ and $0$ otherwise. Here $d_{\mathbb B}$ denotes the hyperbolic metric on $\mathbb B^3.$
\end{lemma}

By Fraser's result ( \cite{JF}, Lemma 3.2), we know that $\rho \left ( p,t \right ) \sim t$ as $t\to \infty ,$ for any point $p\in P.$

%

Now its time to give the proof of Theorem 3.1.

\noindent $Proof~~of ~~Theorem~~ 3.1.$ Let $H^{t}$ be the hyperbolic plane passing through $x_{t}$ normalized by the geodesic line through $O$ and $x_{t}$, where $x_{t}$ is the point in the geodesic ray  $[O, \xi)$ from $O$ toward $\xi\in \mathbb{S}^{2}$ such that the hyperbolic length of the geodesic arc $[O,x_{t}]$ is $t.$  Let $\mu$ denote the Patterson-Sullivan measure of $G.$

For each point $x$ in  $\mathcal{P}=\bigcup_{n= 0}^{\infty } f^{(-n)} (P),$ as Stratmann and  Velani  did in \cite{BO}, we can choose a horoball $  H_{x} $ based at $x$ (a closed Euclidean ball whose interior lies inside $\mathbb{B}^{3} $ and is tangent to $\mathbb{S} ^{2} $ at $x$) such that $H_{x}$ and $ H_{y}$ are pairwise disjoint for different points $x, y$ in $ \mathcal{P}$. Furthermore, the horoballs can be chosen such that $\gamma \left ( H_{x}  \right )= H_{\gamma (x)}  $ for all $\gamma \in G$ and $x\in P.$

We first show the following lemma. The proof method used here is based on Bishop's method in \cite{BJ}.

\begin{lemma}Let $G$ be a non-elementary geometrically finite Kleinian group and $\Lambda(G)$ be its limit set. Let $\mathcal P$ be the set of all parabolic fixed points of $G$. Suppose that
	$$P=\{p_1,\ldots,p_m\}\subseteq\mathcal P$$
	is the finite set of parabolic vertices of a fundamental domain. Then, for each $i=1,\ldots,m$, one can choose $t_i>0$ sufficiently large so that
	\begin{align}K^{*}=\Lambda(G)\setminus
		\left(
		\mathcal P\cup\bigcup_{i=1}^{m}S(p_i,t_i)
		\right)
	\end{align}
	satisfies
	$\mu_\delta(K^{*})>0.$
	Furthermore,
	$\dim_H K^{*}=\dim_H\Lambda(G)=\delta.$
\end{lemma}

\begin{proof} By applying the global measure formula, we know that, for each $x\in \mathcal{P}$, there exists a constant $C$ such that the Patterson-Sullivan measure $\mu _{\delta } $ of the shadow $S(x, t)$  satisfies
	
	$$\mu _{\delta } ( S( x,t))\le C exp\left ( -t\delta -\rho \left ( x,t \right )\left ( \delta -k(x)  \right )   \right )  ,$$
	where $k\left ( x,t \right ) = k\left ( x \right ) $ if $x_{t}\in H_{x}  $ for some $x\in \mathcal{P}$ and $0$ otherwise, and
	$$\rho \left ( x,t \right ) = \inf \left \{ d_{\mathbb{H} }\left ( x_{t},y  \right ) :y\notin H_{x}    \right \} ,$$
	if $x_{t}\in H_{x}  $ for some $x\in \mathcal{P}$ and $0$ otherwise.
	By Fraser's result ( \cite{JF}, Lemma 3.2), we know that $\rho \left ( x,t \right ) \sim t$ as $t\to \infty ,$ for all point $x\in \mathcal{P}.$
	It is easy to see that   $\mu _{\delta }  \left ( S\left ( x,t \right )  \right )\to 0$  by $ exp\left ( -\left ( 2\delta -k \right )t  \right )\to 0 $ as $ t\to \infty .$
	
	Hence, for any $x\in \mathcal{P}$ and positive $r$, there exists a sufficiently large $t_{x} $ such that for any $t>  t_{x} ,$ $\mu _{\delta }  \left ( S(x,t)  \right ) \le r.$ Thus $\mu_\delta(\{p\})=0.$ Note that the set $\mathcal P$ of all parabolic fixed points is countable. Hence $\mu_\delta(\mathcal P)=0,$ since
	$$\mu_\delta(\mathcal P)
	\leq
	\sum_{p\in\mathcal P}\mu_\delta(\{p\})
	=0.$$

	By some easy calculation or see  \cite{JF}, we know that the radius of $S\left ( z,t \right )$ uniformly comparable to $e^{-t} .$ The global measure formula gives $$\mu_\delta\bigl(S(x,t)\bigr)
	\asymp
	e^{-\delta t}
	e^{-(\delta-k(x,t))\rho(x,t)}.$$
	Let \begin{align*}K^{*}=\Lambda(G)\setminus
		\left(
		\mathcal P\cup\bigcup_{i=1}^{m}S(p_i,t_i)
		\right).
	\end{align*}. For each point $p_i\in P$ we can choose a  sufficiently small number $r_{i}$ ,  and  a  sufficiently large number $t_{i}$ such that for $t\geq t_{i},$
	
	$$\mu _{\delta }  \left ( K^{*} \right ) = \mu _{\delta }  \left ( \Lambda(G)\setminus \bigcup_{p_{i}\in P}S(p_i,t_i)
	\right ) \ge 1-\sum_{p_{i} \in P}^{}\mu _{\delta }  \left ( S(p_i,t)  \right )\ge  1- \sum_{p_{i} \in P}r_{i}> 0 .$$
	
	Next we will show $\dim_H K^{*}=\dim_H\Lambda(G)=\delta.$ We consider a Frostman-type estimate.
	
	Notice that $G$ is of divergence type, the associated geodesic flow is ergodic by the Hopf-Tsuji-Sullivan theorem (\cite{DS3}, Theorem 14). Thus, by the ergodic law for the ray excursion function (\cite{DS3}, Corollary 19), we have
	$$\lim_{t\to\infty}\frac{\rho(x,t)}{t}=0
	\quad\text{for }\mu_\delta\text{-almost every }x\in \Lambda(G).$$
	Taking logarithms in the global measure formula and dividing by $(-t),$ we obtain $$\lim_{t\to\infty}
	\frac{\log \mu_\delta\bigl(S(x,t)\bigr)}{-t}
	=\delta.$$
	Setting $l=e^{-t},$ it follows that $$\lim_{r\to0}
	\frac{\log \mu_\delta(S(x,l))}{\log l}
	=\delta$$
	for $\mu_\delta$-almost every $x\in \Lambda(G).$
	
	
	Fixing $s<\delta$, for $\mu_\delta$-almost every $x\in K^{*}$, there exists $l_x>0$ such that for every $0<l<l_x,$
	$$\frac{\log\mu_\delta(S(x,l))}{\log l}>s.$$
	Whenever $0<l<1$, this inequality is equivalent to
	$$\mu_\delta(S(x,l))\leq l^s.$$
	For each positive integer $n$, define
	$$K_n= \left\{ x\in K^{*}:
	\mu_\delta(S(x,l))\leq l^s \text{ for every }0<l<\frac1n
	\right\}.$$
	Note the fact that
	$$\mu_\delta\left(
	K^{*}\setminus\bigcup_{n=1}^{\infty}K_n
	\right)=0$$
	and $\mu_\delta(K^{*})>0$, there exists some positive integer $n$ such that $\mu_\delta(K_n)>0.$ Define the normalized restricted measure as
	$$\nu(A)=
	\frac{\mu_\delta(A\cap K_n)}
	{\mu_\delta(K_n)}.$$
	Then $\nu$ is a probability measure supported on $K_n$.
	
	Fix point $y$. For $0<l<\frac1{2n},$ if $S(y,l)\cap K_n=\varnothing,$ then $\nu(S(y,l))=0.$ Otherwise choose point $\xi\in B(y,l)\cap K_n.$
	Then $B(y,l)\subseteq B(\xi,2l).$ Notice that
	and $\xi\in K_n$, we have
	$$\mu_\delta(S(x,2l))\leq(2l)^s.$$

	Therefore,
	$$\begin{aligned}
		\nu(S(y,l))=
		\frac{\mu_\delta(S(y,l)\cap K_n)}
		{\mu_\delta(K_n)}
		\leq
		\frac{\mu_\delta(S(x,2l))}
		{\mu_\delta(K_n)}
		\leq
		\frac{2^s}{\mu_\delta(K_n)}l^s.
	\end{aligned}$$
	Thus, there exists a constant
	$C_s=\frac{2^s}{\mu_\delta(K_n)}>0$
	such that
	$\nu(S(y,l))\leq C_s l^s$
	for all sufficiently small $l>0$.
	
	Hence by the mass distribution principle (Lemma 3.3), we have $\dim_H K_n\geq s.$
	It follows that $\dim_H K^{*}\geq s,$ since $K_n\subseteq K^{*}.$
	Letting $s$ tend to $\delta$, we have
	$\dim_H K^{*}\geq\delta.$
	On the other hand, $K^{*}\subseteq\Lambda(G)$ shows
	$$\dim_H K^{*}\leq\dim_H\Lambda(G).$$
	
	In conclusion, we have $\dim_H K^{*}=\delta=\dim_H\Lambda(G).$
\end{proof}

Now, we continue to prove Theorem 3.1.

Suppose $g$ is a quasiconformal homeomorphism of $\mathbb{B}^{3}$ which is  compatible with a geometrically finite  Kleinian group $G$. Then the quasiconformal conjugation group $G'=gGg^{-1}$ is also a geometrically finite Kleinian group, and $g(\Lambda(G))=\Lambda(G').$
Moreover, $g$ preserves the types of limit points. Hence $\mathcal P'=g(\mathcal P)$
is the set of all parabolic fixed points of $G'$.

Thus, $P'=g(P)$ is a finite set of  the $G'$-orbits inequivalent parabolic fixed points. For each $p\in P$, we can choose a shadow neighborhood $V_{g(p)}$ of $g(p)$ as  $V_{g(p)}=S\bigl(g(p),t'_{g(p)}\bigr)$
such that
$$L=
\Lambda(G')\setminus
\left(
\mathcal P'\cup\bigcup_{p\in P}V_{g(p)}
\right)$$
with $\dim_H L=\dim_H\Lambda(G')=\delta'.$

Since $g$ is a homeomorphism, $g^{-1}(V_{g(p)})$ is also a neighborhood of $p$. Let $t$ tend $\infty$, the shadows $S(p,t)$ form a neighborhood basis at $p$. We  choose $t_p$ sufficiently large so that
$$U_p=S(p,t_p)\subset g^{-1}(V_{g(p)}).$$
Consequently, $g(U_p)\subset V_{g(p)}.$
Define $K$ as
$$K=
\Lambda(G)\setminus \left(\mathcal P\cup\bigcup_{p\in P}U_p
\right).$$
We have
$$g(K)
=\Lambda(G')\setminus
\left(\mathcal P'\cup\bigcup_{p\in P}g(U_p)\right).$$
Since $g(U_p)\subset V_{g(p)}$, it follows that
$L\subset g(K).$
Therefore $$\dim_H g(K)\geq \dim_H L=\delta'.$$
On the other hand, $$g(K)\subset g(\Lambda(G))=\Lambda(G').$$
Thus  $$\dim_H g(K)\leq\dim_H\Lambda(G')=\delta'$$
and
$$\dim_Hg(K)=\dim_H\Lambda(G').$$ \qed

It is known that for a geometrically finite  Kleinian group $G$,  the limit set $\Lambda (G)$ is exactly the union of the conical limit set $\Lambda _{c} \left ( G \right ) $ and the parabolic fixed points set $\Lambda_{P}(G)$ which contains countably many parabolic fixed points.
The parabolic fixed points are just corresponding to the parabolic periodic points of the Markov map $f$ and the inverse orbits of the parabolic points.
Recall that a point $x$ is called strictly preperiodic of the map $f$ if there exist some positive numbers $i$ and $j$ such that $f^{i}(x)=f^{i+j}(x).$
The subset  $\mathcal{P}^{*}=\bigcup_{n= 1}^{\infty } f^{-n} P$ of $\mathcal{P}= \bigcup_{n= 0}^{\infty } f^{-n} P$ is just the preperiodic points set of inequivalent parabolic periodic points set $P$ for $f$, that is, $P\subset \overline{\mathcal{F}} \cap \mathbb{S}^{2} .$

Assume that $\mathcal{F}$ is an even-cornered fundamental domain for $G,$ and $\mathfrak{R } $ is a Markov partition of $G$ corresponding to $f$ on $\mathbb{S}^{2} .$
For each point $p\in P$ with rank 1, there are two faces corresponding to a vertex $p$  with different dynamical behavior: the concentrate face and pervasion face respectively. We will refer to both faces cusp-accesses, see FIGURE 2.1 . Similarly, each point $p\in P$ with rank 2 has four one-faced accesses.

Each access is chosen to be a repelling petal of the corresponding periodic return map. To avoid ambiguity between different Markov branches, we will work in the access-splitting space (or Markov extension space), denoted by $\mathcal{AP}$. Thus, if the point $p\in P$ belongs to  different accesses (belongs to different branches of $f$), we represent it by distinct copies $p_{a} .$ Consequently, distinct points $p_{a} $ and $p_{b} $  in the Markov accesses space $\mathcal{AP}$ may have the same projection onto the limit set. We must point out that a access corresponding to different branches must be treated separately. Even though they may project to the same point in the  limit set, they are regarded as distinct boundary points in the access space.

For each $p_a\in \mathcal{AP}$ corresponding to a Markov branch, we can choose a sufficiently small one-sided neighborhood $V_{p_a}$ contained in the repelling petal of that branch. Consider a subset  of the parabolic points set $\mathcal{P}.$
Let
\begin{align}
	S= \left \{ p\in \partial \mathfrak{R}:\ p\text{ is preperiodic to a parabolic periodic point  }  \right \},
\end{align}
which contains all the parabolic points on the boundary of the Markov partition $\mathfrak{R}$.
Thus $V_{p_a}$ is a one-sided interval (or a one-faced access neighborhood). Each preperiodic point $p\in S$ is corresponding to the branch lying on its boundary $p\in\partial V_{p_a}.$

Now we generalize Bowen's results \cite{IM} from the  interval to the two-dimensional sphere $\mathbb{S}^{2} $ as follows.

\begin{lemma}Suppose $G$ is a geometrically finite Kleinian group acting on the hyperbolic unit ball space $\mathbb{B}^{3}$ with parabolic elements and $f$ is the Markov map of $G.$ Then for any preperiodic point $p\in S$, there exists a sufficiently small neighborhood $V_{p} $ of $p$ such that for all $z\in V_{p} \setminus \left \{ p \right \} ,$ the constant $A\left ( V_{p}  \right )$ defined by  \begin{align}A(V_{p})= \inf_{ z\in V_{p}} \frac{\left | f^{'} \left ( z \right )  \right | _{\sigma }  \, |z-p|_{\sigma}}{|f(z)-f  (p)|_{\sigma}}\end{align} is positive.
	
\end{lemma}

\begin{proof}Let $p\in S$ be a preperiodic point. By the properties of the Markov accesses associated with points in $S$, the corresponding branch of the Markov map $f_G$ can be extended holomorphically  to a sufficiently small neighborhood $V_p$ of $p$.
	
	By composing some conformal homeomorphisms $\gamma$ and $\eta$ of the unit ball $\mathbb{B}^{3}$, we can normalize $f$ by $F\left ( \xi \right )= \eta \circ f\circ \zeta ^{-1} \left ( \xi \right )$ with $F(0)=0.$ Let $k\geq 1$ be the order of the zero of $F$ at $0$. Taylor expansion of $F$ at $0$ as $$F\left ( \xi \right )= b_{k} \xi^{k} +O\left ( \xi^{k+1}  \right )  ,$$ where $b_{k}=F^{(k)}(0)/k!.$
	Observe that $$F^{'} \left ( \xi \right )= kb_{k}\xi^{k-1} +O\left ( \xi^{k}  \right ).$$
	For sufficiently small $\left | \xi \right |$ ,we have  $\left | F\left ( \xi \right )  \right |\sim \left | \xi \right |^{k}   $ and $\left | F^{'} \left ( \xi \right )  \right |\sim \left | \xi \right |^{k-1},$ and
	\begin{align}
		\frac{\left | F^{'} \left ( \xi \right )  \right |\left | \xi \right | }{\left | F\left ( \xi \right )  \right | }\sim 1 .\end{align}
	Hence, for sufficiently small $\left | \xi \right |,$
	\begin{align}\left | F^{'}\left ( \xi \right )   \right | \ge C \frac{\left | F\left ( \xi \right )  \right | }{\left | \xi \right | }, \end{align} where $C$ is a positive constant.
	
	Let $f\left ( p \right ) =q.$ Note that the spherical distance  $\left | z- p \right | _{\sigma }$ is equivalent to the Euclidean distance $ \left | \xi -0\right | $ in a  sufficiently small coordinate neighborhood of $p$,  we have that \begin{align}\left | f\left ( z \right ) - q \right | _{\sigma }\sim \left | F\left ( \xi  \right )   \right |\end{align}
	and the spherical derivative is equivalent to the local coordinate derivative \begin{align}
		\left | f^{'}\left ( z \right )   \right | _{\sigma } \sim \left | F^{'}\left ( \xi  \right )   \right | .
	\end{align}
	Furthermore, for any $z\in V_{p} \setminus \left \{ p \right \} ,$  we can find a positive constant $C$ such that
	\begin{align}\left | f^{'}\left ( z \right )   \right | _{\sigma } \ge C \frac{\left | f\left ( z \right ) - q \right |_{\sigma }  }{\left |z-p \right |_{\sigma }  } .\end{align}
	
	In conclusion, for any point $p\in S $, by(3.6-3.10), there exists a neighbourhood $V_{p}$ such that for all $z\in V_{p} \setminus \left \{ p \right \} ,$ we have $$A(V_{p})= \inf_{ z\in V_{p}} \frac{\left | f^{'} \left ( z \right )  \right | _{\sigma }  \, |z-p|_{\sigma}}{|f(z)-f  (p)|_{\sigma}} \geq C> 0.$$
\end{proof}

Let $p\in P$ be a  parabolic periodic point of $f$ with period $r$ and $V_{p}$ be a sufficient  small neighbourhood of $p,$ which implies $\lim_{k\to\infty} (f^{-r})^{k} (z) = p$ for any $z\in V_{rep, p} .$
The parabolic petal theorem(see \cite{MJh})implies that: Within any neighborhood of $p,$ attracting petals and repelling petals are alternating appearance.
Hence we can divide $V_{p}$ into attracting petals region and repelling petals region, denoted by $V_{att, p}$ and $V_{rep, p}$ separately.
We point out that the point $p$ is on the boundary of $V_{att, p}$ and $V_{rep, p}.$

For points $z\in V_{rep, p},$
there exists some $m\ge 1$ such that the $m$th forward image $f^{rm} \left ( z \right ) $ lies outside of $V_{p},$ which implies $\left | \left ( f^{r}  \right )^{'}  \right |_{\sigma }  > 1$ on $\overline{V}_{rep, p} \setminus \left \{ p \right \} .$
Hence for any point $z\in V_{rep,p} \setminus \left \{ p \right \} ,$  we can define  $m_{V}\left ( z \right )= \min \left \{ m\ge 0:f^{rm}\left ( z \right )\notin V_{rep, p} \right \} .$

\begin{lemma}Let $G$ be a geometrically finite Kleinian group with cusps.  Let $ P_{1}$ be the subset of Markov access space $\mathcal{AP}$ which contains all the parabolic period point of $f.$  Then for any $p\in P_{1}$ with period $r,$ there exists a sufficiently small repelling petal region $V_{rep, p}  $ of $p$ and a positive constant $B\left ( V_{rep,p}  \right ) $ such that
	$$\left | \left ( f^{rm_{V}\left ( z \right )  }  \right )^{ '}\left ( z \right )   \right | _{\sigma }> \frac{B\left ( V_{rep, p}   \right ) }{\left | z- p \right |_{\sigma }  } \:\:\:\  for  \:all \:z\in V_{rep,p} , $$ where $m_{V}\left ( z \right )= \min \left \{ m\ge 0:f^{rm}\left ( z \right )\notin V_{rep,p}   \right \} .$
\end{lemma}

For the simplify, we will use the hyperbolic upper half plane space $\mathbb{H}^{3}$ to prove the  lemma.
\begin{proof} Suppose  $p$ is an $r$-periodic point of $f,$ i.e. $p=f^{r}\left ( p \right )$ is a fixed point of $f^{r}.$
	Note that $p$ is also a parabolic fixed point of the Kleinian group $G.$ By conjugating $G$ with some M\"obius transformations,  we  can assume $p=\infty$  and the  parabolic element $T_{\lambda }(z)= z+\lambda,\: \lambda\ne 0 $ fixing $p$ is in $G.$

	By the construction of the Markov map and the definition of the accesses corresponding to $p$, after shrinking the cusp neighborhood, the return map branch $f^r$ coincides with $T_\lambda^{-1}$.  Hence for sufficiently large $|z|$, we have $f^{r}(z)=T^{-1}(z).$
	By the inversion transformation $\eta( u)=\frac{1}{u},$ the local behaviour of $f^{r}$ near $p$ can be represented as the local property of $F=\eta\circ f^{r}\circ\eta^{-1}(u)$ near $0.$ By some easy calculation, we know   $F(u)=\frac{u}{1-\lambda u} $ in a neighbourhood $0.$
	The Taylor expansion of the parabolic rational mapping $F$ near $0$ is $$F\left ( u \right ) = u+\lambda u^{2} +O\left ( u^{3}  \right ) .$$
	Hence $u$ has one attracting petal region and one repelling petal region near $0$.
	
	We can have a sufficiently small neighborhood
	\begin{align} V_{R,H} = \left \{ u:\mathfrak{Re}(\alpha(u))> R,\left | \mathfrak{Im}(\alpha(u))  \right |< H   \right \}
	\end{align} of $0$ in the repelling petal of $F$ by taking sufficiently large $R$ and $ H$
	to ensure that $V_{R, H}\cap K= \emptyset $ and for different $p,q\in P,$ $\overline{V}_{rep,p} \cap \overline{V}_{rep,q} = \emptyset ,$	
	where $\mathfrak{Re}(\alpha(u))$ denotes the real part of $\alpha(u)$ and $\mathfrak{Im}(\alpha(u))$ denotes the imaginary part of $\alpha(u).$
	It is easy to see that $V_{R,H}$ is a sufficiently small region of $0,$ denoted by $V_{rep,0}=V_{R,H}.$
	According to the Parabolic Linearization Theorem (see \cite{MJh}), we know that there is one and, up  to composition with a translation of $\mathbb{C}$, only one conformal embedding $\alpha:V_{rep, 0}\rightarrow \mathbb{C}$ which satisfies the Abel function equation $\alpha(F^{(-1)}(u))=1+\alpha(u).$ A simple calculation yields $\alpha(u)=\frac{1}{\lambda u}.$
	We call $\alpha(u)$ the Fatou coordinate of $u.$ It is easy to see that the Fatou coordinates satisfy $\alpha(\left ( F^{(-m)}\left ( u \right )   \right )) =\alpha(u)+m .$

	Note that for $u\in V_{rep,0},$ $F$ is expressed in Fatou coordinates satisfying  \begin{align}
		\alpha\left( F^{m}\left ( u \right )   \right ) =\alpha( u )-m .
	\end{align}
	Hence, let $m_{V}\left ( u \right )= \min \left \{ m\ge 0:F^{m}\left ( u \right )\notin V_{rep,0}    \right \} .$ By translation doing not change the imaginary part, we know that $m=m_{V}$ is the first time at which the trajectory leaves $ V_{rep,0}  ,$ so we have: $$\mathfrak{Re} \left ( \alpha(u) -\left ( m_{V}-1 \right )  \right ) > R \:\:and\:\:\mathfrak{Re}\left ( \alpha(u)-m_{V} \right ) \le R .$$
	Note that $\left \{ u:R-1\le \mathfrak{Re}(\alpha(u))\le R, \left | \mathfrak{Im}(\alpha(u)) \right |   \le H   \right \}  $ is a compact set which is away from  $0.$ Consequently, the derivatives of $\alpha$ on this compact set  have an upper bound. Thus there exists  some $0< C_{1}<\infty  $ such that
	\begin{align}\left | \alpha^{'}\left ( F^{m_{V}}  \right )   \right |\le C_{1}  .\end{align}
	Since $F(u)=\frac{u}{1-\lambda u}$ in a neighbourhood $0,$ we obtain
	$F^{m_{V}} \left ( u \right ) = \frac{u}{1-m_{V}\lambda u} $  for $u$ near $0.$
	Furthermore, note that $$\alpha( F^{m_{V}}\left ( u \right )  ) =\alpha( u )-m .$$
	Thus \begin{align}
		\alpha^{'} ( F^{m_{V}} \left ( u \right )   )\left ( F^{m_{V}}   \right )^{'}\left ( u \right ) = \alpha^{'}(u).
	\end{align}
	We have
	\begin{align} \left ( F^{m_{V}}   \right )^{'}\left ( u \right ) =  \frac{ \alpha^{'} \left ( u \right )  }{\alpha^{'} \left ( F^{m_{V}} \left ( u \right )  \right )}.
	\end{align}
	By the asymptotic properties of Fatou coordinates: $\left | \alpha ^{'}\left ( u \right )   \right | \sim  \frac{1}{\lambda}\left| u \right |^{-2}  .$
	
	According to (3.13-3.15), there exists $C_{2}> 0$ such that $$\left | \left ( F^{m_{V}}  \right )^{'}\left ( u \right )    \right | \ge C_{2} \left | u \right | ^{-2} .$$

	For any $z\in V_{rep,p},$ observe that the transformation $z=\eta(\left ( u \right )) = 1/u$ and the spherical metric $\sigma = 2\left | \mathrm{d}z \right | /\left ( 1+\left | z \right | ^{2}  \right ) $, we have
	\begin{align}\left | z-p \right | _{\sigma }\sim \left | u \right | \end{align}
	and
	\begin{align}\left | \left ( f^{rm_{V}}  \right ) ^{'}\left ( z \right )   \right | _{\sigma }= \frac{\sigma \left ( F^{m_{V}}\left ( u \right )   \right ) }{\sigma \left ( u \right ) }\left | \left ( F^{m_{V}}  \right )^{'} \left ( u \right )   \right |.
	\end{align}
	Note that $u=\eta^{-1}(z)$ ,there exists a constant $C_{3} > 0$ such that $$\frac{\sigma \left ( F^{m_{V}}\left ( u \right )   \right ) }{\sigma \left ( u \right ) }\ge C_{3} $$
	and
	\begin{align}\left | \left ( f^{rm_{V}}  \right ) ^{'}\left ( z \right )   \right | _{\sigma }\ge C_{3}\left | \left ( F^{m_{V}}  \right )^{'} \left ( u \right )   \right |\ge C_{3}C_{2}\left | z-p \right | _{\sigma }^{-2}   .\end{align}
	Since $V_{\mathrm{rep},p}$ can be chosen sufficiently small, there exists $\delta>0$ such that $$|z-p|_{\sigma}\leq \delta
	\qquad\text{for all }z\in V_{\mathrm{rep},p}.$$
	Hence,  for  $\:all \:z\in V_{rep,p} $, let $B\left ( V_{rep,p}  \right )=C_{3}C_{2}/\delta> 0 $, we obtain $$\left | \left ( f^{rm_{V}\left ( z \right )  }  \right )^{ '}\left ( z \right )   \right | _{\sigma }> \frac{B\left ( V_{rep,p}   \right ) }{\left | z- p \right |_{\sigma }  }. $$
	
\end{proof}

The following lemma may be well known by experts, however we do not find any references, so we give a short proof here.

\begin{lemma} Suppose $p\in P$ is a parabolic periodic point of $f$ with period $r$ and $$\mathcal{O} \left ( p \right ) = \left \{ p,f\left ( p \right ),...,f^{r-1} \left ( p \right )   \right \} $$ denotes its periodic orbits. Let $V_{rep,p} $ be a repelling petal neighborhood of $p.$ Then $$\bigcap_{L=0}^{\infty } \left \{ x\in \overline{V}_{rep,p} :f^{jr}\left ( x \right ) \in \overline{V}_{rep,p} ,0\le j\le L    \right \}= \left \{ p \right \}  .$$ Moreover,for any $\varepsilon > 0$, there exists a finite constant $L_{0}\ge 1 $  such that any point in $$C_{\varepsilon } = \left \{ x\in \overline{V}_{rep,p} :\left | x- p \right |_{\sigma }  \ge \varepsilon  \right \} $$ must leave $\overline{V}_{rep,p} $ after at most $L_{0} $ iterations.
\end{lemma}

\begin{proof} Suppose $p\in P$ is a repelling periodic point and $V_{rep,p}$ is a repelling petal of $p.$ For every point $x\in \overline{V}_{rep,p} \setminus \left \{ p \right \} ,$ there exists some $L\ge 1$ such that its $L$-th forward image $f^{rL} \left ( x \right ) $ lies outside $V_{rep,p}.$
	Denoting by $F= f^{r} ,$ then for every $x\in \overline{V}_{rep,p} \setminus \left \{ p \right \} ,$ we have  $\left | F^{'} \left ( x \right )  \right | _{\sigma } > 1$ and  $$\left | F\left ( x \right )-p  \right |_{\sigma } =  \left | F\left ( x \right )-F\left ( p \right )  \right |_{\sigma }= \left | F^{'}\left ( \xi  \right )   \right | _{\sigma }\left | x- p \right |_{\sigma }> \left | x- p \right |_{\sigma }.$$
	Note that $F\left ( p \right ) = p,$ we have $$ p  \in \left \{ x\in \overline{V}_{rep,p}  :F^{j}\left ( x \right ) \in \overline{V}_{rep,p} \:\:for\:\:all\:\:j\ge 0  \right \} .$$
	
	Suppose there exists  another point $x\ne p$ in $\overline{V}_{rep,p}  $ such that $F^{j} \left ( x \right ) \in \overline{V}_{rep,p}  $ for all $j\ge 0.$
	Let $d_{j}= \left | F^{j} \left ( x \right )-p  \right | .$ It is easy to see that $d_{j+1}> d_{j} .$
	Furthermore, for any positive integer $j,$  $F^{j} \left ( x \right )$ belong to the compact set $\overline{V}_{rep,p} ,$ we have $$d_{j}\le \max_{x\in \overline{V}_{rep,p} } \left | x- p \right |_{\sigma } < \infty   .$$ Thus the sequence $\{d_{j}\}$ has a limit $d$ as $j\rightarrow\infty.$
	The monotonicity of $d_{j}$ implies that $d\ge d_{0} = \left | x- p \right | _{\sigma } > 0.$
	
	Observe that $\overline{V}_{rep,p}  $ is compact, for the point $x\in \overline{V}_{rep,p} $, $\{F^{j}(x)\}$ has a convergent subsequence $\{F^{j_{l}}(x)\}$ which converge to some  $y\in \overline{V}_{rep,p}  .$
	Thus $$\left | y-p \right |_{\sigma } =\lim_{l \to \infty}\left | F^{j_{l} }\left ( x \right )-p   \right | = d> 0 ,$$ thus $y\ne p.$
	Therefore we have $$\left | y-p \right | _{\sigma } =\lim_{l \to \infty}\left | F^{j_{l+1} }\left ( x \right )-p   \right |=\lim_{l \to \infty}\left | F\left ( F^{j_{l} }\left ( x \right )  \right )-p   \right |  =\left | F\left ( y \right ) -p \right |_{\sigma } ,$$
	
	However, by the expansibility of points in $V_{rep,p}$,    $\left | F\left ( y \right ) - p \right | _{\sigma }>  \left | y-p \right | _{\sigma } .$ This leads a contradiction. Thus there no such point $x\in \overline{V}_{rep,p} \setminus \left \{ p \right \}  $ such that whose forward $F$-orbit remains entirely in $\overline{V}_{rep,p} .$ Consequently, the set $$ \left \{ x\in \overline{V}_{rep,p}  :F^{j}\left ( x \right ) \in \overline{V}_{rep,p} \:\:for\:\:all\:\:j\ge 0  \right \} $$ contains only a point $p.$
	
	Assume that there is $\varepsilon_{0}>0,$  for any  $n\ge 1,$ there exists a point $z_{n} \in C_{\varepsilon } , $ such that $$z_{n} ,F\left ( z_{n} \right ),...,F^{n}\left ( z_{n} \right )\in \overline{V}_{rep,p}   .$$
	Observe that $C_{\varepsilon }$ is compact, after passing to a subsequence, we may assume that $z_{n}\to z_{\ast }.$ So $$\left | z_{\ast }  -p \right |_{\sigma }  \ge \varepsilon > 0.$$
	
	Fix any $j\ge 0.$ Note that $F^{j} \left ( z_{n}  \right ) \in \overline{V}_{rep,p} $ for any positive integer $n.$
	Since $z_{n}\to z_{\ast } $ as $n\rightarrow +\infty,$ it follows that
	$$F^{j} \left ( z_{n}  \right )\to   F^{j} \left ( z_{\ast }  \right )\in\overline{V}_{rep,p}.$$
	We have $F^{j}\left ( z_{\ast }  \right )\in \overline{V}_{rep,p}$ for all $j\ge 0$. However, the only point whose forward orbit can remain in $\overline{V}_{rep,p}$ forever is $p.$ Hence $z_{\ast } = p,$ which contradicts the fact that $z_{\ast} \in C_{\varepsilon } $ and $\left | z_{\ast }-p   \right |_{\sigma }  \ge \varepsilon > 0$  This contradiction leads to a conclusion of the lemma.
\end{proof}


If $G$ is a co-compact Kleinian group, Rocha \cite{R} had proved that the corresponding Markov map $f$ corresponding to $G$ is (eventually) expanding on $\Lambda(G)$.
Let $K$ be a subset of $\mathbb{S}^{2}$ constructed as Theorem 3.1. It is well known that if there exist some $N\ge 0$ such that, for any $y\in \Lambda \left ( G \right ) $ and  $n\ge N$,  $f^{n} \left ( y \right ) \notin K,$ then $y$ is a parabolic fixed point of $G.$
For each point $x\in K,$  let $m\left ( x \right ) $ denote the smallest $m\ge 1$ such that $f^{m} \left ( x \right ) \in K,$ and $f_{K} \left ( x \right ) = f^{m\left ( x \right ) }( x).$ We will extend  Bowen and Series' work \cite{BS} of Fuchsian groups  and Rocha's work of co-compact Kleinian in \cite{R} to the case of  Kleinian groups with parabolic elements. We restate Theorem 1.1 by the model of  the hyperbolic unit ball space  $\overline{\mathbb{B}}^{3}$ as follows.
\begin{theorem}
	Let $G$ be a geometrically  Kleinian group acting on the hyperbolic unit ball space  $\mathbb{B}^{3}$ with parabolic elements. Then for the  subset $K$ constructed in Theorem 3.1, the first return map $f_{K}$ induced by the Markov map $f_{G}$ on $K$ is (eventually) expanding.
\end{theorem}
Drawing on the method from Bowen(\cite{IM}, Lemma 3), we have

\begin{proof}
	To prove the theorem we only need to prove the existence of universal constants $M_{0}$ and $C> 1$ such that $\left | \left ( f^{M_{0}}_{K}   \right )^{'} \left ( x \right )   \right |_{\sigma }   \ge  C.$
	
	For every point $x\in K,$ we let $m\left ( x \right ) $ denote the smallest $m> 1$ such that $f^{m} \left ( x \right ) \in K,$ and $f_{K} \left ( x \right ) = f^{m\left ( x \right ) } (x).$ Let  $\mathfrak{R}$ be the Markov partitions for the group $G.$  Recall that $P$ is the set of all parabolic periodic points or a point in the periodic orbit of some parabolic periodic point of $f$. Let $S$ be the set of all strictly pre-parabolic periodic points in $\partial\mathfrak{R}.$ For convenience, we split the parabolic points in \(S\) and \(\mathcal P\) according to their Markov accesses. Thus, the same geometric point corresponding to different Markov accesses is regarded as two distinct points.
	Note that $f\left ( S\cup P \right ) \subset S\cup P$ and $\left ( P\cup S \right ) \cap K= \emptyset.$ For any $x\in \Lambda _{c} \left ( G \right )$  and for $\eta>1,$ there exists a positive integer $t(x)$ such that the conformal derivative of $f^{t\left ( x \right ) }$ is bigger than $\eta.$ Denoted by $g\left ( x \right )= f^{t\left ( x \right ) }$, we have $\left | g^{'} \left ( x \right )  \right | _{\sigma } \ge \eta > 1.$
	
	Fix a real number $C>1$. For $x\in K$, let $M(x)$ be the smallest positive integer such that the conformal derivative of $f_{K}^{M(x)}(x)$ is bigger than $C$. The existence of such $M(x)$ will be given latter.
	
	In the following, we will  use $g\left ( x \right )$ to partition the $f$ orbits between $x$ and  $f_{K}^{M(x)}\left ( x \right ) $ as
	$$x,\; g(x),\; g^{2}(x),\; \dots,\; g^{j}(x),\; f_{K}^{M(x)}(x),  g^{j+1}(x)\dots.$$
	Here $j$ is the positive integer such that $f_{K}^{M(x)}(x)$ lies between $g^{j}(x)$ and $g^{j+1}(x).$

	Let $\alpha = \sup\{ \left | f^{'} \left ( y \right )   \right |_{\sigma }   : y\in \Lambda (G)  \} $ which is finite because $\Lambda (G)$ is compact and $f$ is piecewise smooth. Let $N$ be the number of $f-$orbits of $x$ between $f_{K}^{M(x)}\left ( x \right )$ and $g^{j+1}(x)$(later we will show that $N$ is bounded by a constant which does not depend on $x$). Note that $\left | g^{'}\left ( x \right )   \right |_{\sigma } \ge \eta> 1 .$ Then we have
	$$\left | \left ( f^{M(x)}_{K}   \right )^{'} \left ( x \right )   \right |_{\sigma }   \ge \eta^{\,j+1} \cdot \alpha^{-N} \ge \eta^{\,M(x)/N} \cdot \alpha^{-N}.$$
	Hence, for any $x\in K$ and for some $C > 1$, there must exists a sufficiently large $M$ such that $$\left | \left ( f^{M}_{K}   \right )^{'} \left ( x \right )   \right | _{\sigma }  \ge  C > 1,$$  and denotes  the smallest such integer  number $M$ by $M(x).$
	Since for all $0\le k\le j,$ the $f-$orbits of $x$ from $g^{k} \left ( x \right ) $ to $g^{k+1} \left ( x \right ) $ enter $K$ at most $N$ times, and $f_{K}^{M(x)}\left ( x \right ) \in K ,$ we obtain $M(x) \le N(j+1).$
	We only need to show that $M(x)$  and $N$ are bounded constants which do not depend on $x\in K.$
	To do so, just as McMullen did  in \cite{Mc3}, we will divide $\Lambda \left ( G \right ) $ into to two subsets $K_{1}$ and $K_{2}$.
	
	First, for each parabolic periodic point $v\in P,$ we take  a repelling petal neighborhood $V_{v} =V_{rep,v}  $ as the proof of Lemma 3.6 does. For any parabolic pre-periodic point $p\in S,$ let $V_{p} $ be a neighbourhood of $p$ satisfying the conditions of Lemma 3.5, and additionally $f\left ( V_{p}  \right ) \subset V_{f\left ( p \right ) } $.
	
	Let $ K_{1}= \bigcup_{p\in S\cup P }^{} V_{p} $ be a subset of $\Lambda(G).$  For every $p\in S\cup P$  and $x\in V_{p},$ let $v = f^{n(p)}(p),$ where $n(p)$ is the smallest integer such that $v=f^{n}(p)$ is a periodic of $f$ with order $r$. Then $f^{n(p)} (x) \in V_{v}$ and let $m=m_{V} (x)$ be the smallest positive integer such that $f^{rm}(f^{n(p)} (x)) \notin V_{v}.$ We have
	
	\begin{align}\left | \left ( f^{rm(y)+n(p)}  \right )^{'} (x)  \right |_{\sigma } = \left | \left ( f^{rm(y)}  \right )^{'} (f^{n(p)}(x))  \right |_{\sigma }\prod_{k=0}^{n(p)-1} \left | f^{'} (f^{k}(x))  \right |_{\sigma }. \end{align}

	By the Markov property of $f$, for any $x\in V_{p} ,$ we have $f^{n\left ( p \right ) }(x)\in f^{n\left ( p \right ) }(V_{p}) (\subset V _{v} )$. Furthermore, by Lemma 3.6, there exists a constant $B(V_{v})$ such that
	
	\begin{align}\left | \left ( f^{mr}  \right )^{'}\left ( f^{n\left ( p \right ) }(x) \right )    \right | _{\sigma }  \ge \frac{B(V_{v})}{|f^{n\left ( p \right ) }(x)-v|_{\sigma}}. \end{align}
	
	For any $k\in \left [ 0,n\left ( p \right )- 1  \right ] $, by Lemma 3.5, we  obtain
	
	\begin{align}A(V_{f^{k} (p)})= \inf_{ f^{k} (x)\in V_{f^{k}( p)}} \frac{\left | f^{'} \left ( f^{k} (x) \right )  \right | _{\sigma }  \, |f^{k} (x)-f^{k}( p)|_{\sigma}}{|f^{k+1}(x)-f^{k+1}  (p)|_{\sigma}} > 0. \end{align}

	Substituting (3.20) and (3.21) into (3.19), we obtain
	\begin{align*}&\left | \left ( f^{rm}  \right )^{'}\left ( f^{n(p)}(x) \right )  \right | _{\sigma }\prod_{k=0}^{n(p)-1} \left | \left ( f^{}  \right )^{'}\left ( f^{k}(x) \right )   \right | _{\sigma }\\
		&\ge \frac{B(V_{v})}{|f^{n(p)}(x) - v|_{\sigma}} \prod_{k=0}^{n(p)-1} \frac{A(V_{f^{k}(p)})\, |f^{k+1}(x) - f^{k+1}(p)|_{\sigma}}{|f^{k}(x) - f^{k}(p)|_{\sigma}}\\
		&\ge \frac{B(V_{v})}{|x-p|_{\sigma}} \prod_{k=0}^{n(p)-1} A(V_{f^{k}(p)}).
	\end{align*}
	
	If $n(p)=0,$ then $p=v$ and for any $x\in \overline{V_{v}} -\{v\}$, $\left | \left ( f^{rm}  \right )^{'}\left ( x \right ) \right |_{\sigma } > C_{2}$ for some constant $C_{2}>1$ because $\left | (f^{r} )^{'}(x) \right |_{\sigma } > 1$ on $\overline{V_{v}} -\{v\}.$

	If $n(p)>0,$ that means $p$ is an orbital point of $v.$
	We take the diameter of $V_{p} $ to be sufficiently small so that the following inequality holds
	$$\operatorname{diam} V_{p} < A(V_{p}) A(V_{f(p)}) \dots A(V_{f^{n(p)-1}(p)})\, B(V_{f^{n(p)}(p)})$$
	and the quantity $\left | \left ( f^{rm+n(p)}  \right )^{'} (x)  \right |_{\sigma }$ is bigger than some constant $\eta>1.$ Hence, for any $x\in V=\bigcup_{p\in S\cup P} V_{p} ,$  we can define $g(x) = f^{t(x)}=f^{rm(x)+n(p)} (x).$

	In the following we will show that there is an integer $N_{1} $ independent of $x$ so that if $f^{k} (x)\in K$ is on the $f$-orbit between $x$ and $g\left ( x \right ), $ then $rm+n\left ( p \right )< k+N_{1}.$
	
	Observe that if the orbit of a point contains sufficiently many complete period blocks before leaving a neighborhood of the periodic orbit, then that point must be arbitrarily close to the periodic orbit.
	
	Suppose the periodic orbit of a period  point $v$ is $$\mathcal{O} \left ( v \right ) = \left \{ v,f\left ( v \right ),...,f^{r-1} \left ( v \right )   \right \}. $$
	If $x,F(x),\ldots,F^L(x)\in V_v,$ then by Lemma 3.7, for any $\varepsilon>0$ ,  there is  a sufficiently large positive number  $L$ such that $\left | x- v \right |_{\sigma }  < \varepsilon,$ where $F=f^r$.
	More precisely, let $\varepsilon>0$ be given. For $0\le s<r,$  since $f^s$ is continuous at $v,$ there exists $\delta_s>0$ such that if $|x-v|_{\sigma } <\delta_s
	,$ we have $$|f^s(x)-f^s(v)|_{\sigma } <\varepsilon.$$
	Because there are only finitely many phases $s=0,\ldots,r-1,$ we may take $$\delta=\min_{0\le s<r}\delta_s>0.$$
	Therefore, for every $s\in\{0,\ldots,r-1\},$ if $|x-v|_{\sigma } <\delta
	,$ we have $$|f^s(x)-f^s(v)|_{\sigma } <\varepsilon.$$
	By Lemma 3.7, there exists a positive integer  $L=L(\delta,v)$ such that, whenever $x,F(x),\ldots,F^L(x)\in\overline{V_v},$ we have
	$|x-v|_{\sigma } <\delta.$
	Consequently, $$|f^s(x)-f^s(v)|_{\sigma }<\varepsilon.$$
	Note that $f^s(v)\in P,$ it follows that $$dist(f^s(x),P)
	\le |f^s(x)-f^s(v)|_{\sigma }
	<\varepsilon.$$
	
	We now prove that $N_{1} $ must be finite. First, from the condition $f(V_p)\subset V_{f(p)},$ we know that the orbit of any preperiodic point in $p\in S$ cannot enter $K$ before reaching the periodic orbit.
	
	%
	For any $x\in V_{p},$ recall that $m=m_{V} (x)$ is the smallest positive integer such that $f^{rm}(f^{n(p)} (x)) \notin V_{v}.$ Let $z= f^{n\left ( p \right ) } \left ( x \right ) \in V_{v} $.
	We have $$F^{j} \left ( z \right )= f^{rj} \left ( z \right ) \in V_{v}  \:\:for \:\:all\:\:0\le j< m,$$ where $F= f^{r} .$ Let $k= n(p)+rl+  s,$ where $0\le s< r,$ and suppose that $$f^{k} \left ( x \right ) = f^{s}\left ( F^{l}\left ( z \right )   \right ) \in K .$$

	Since the orbit does not leave $V_{v} $ until the $m$-th periodic block, starting from $F^{l} \left ( z \right ) ,$ there are at least $m-l$ periodic blocks remaining in $V_{v} .$ That is $$F^{l} \left ( z \right ) ,F^{l+1} \left ( z \right ),...,F^{m-1} \left ( z \right )\in V_{v} .$$
	
	Assume that there is no finite number $N_1,$  Then,  we have $$f^{k} \left ( x \right ) = f^{s}\left ( F^{l}\left ( z \right )   \right ) \to f^{s}\left ( v \right )\in P   $$
	as $k$ tends to infinity.
	
	Hence $dist\left ( f^{k}\left ( x \right ),P   \right ) \to 0.$ Since $\bar{K} $ is compact and $\bar{K} \cap \left ( P\cup S \right ) = \emptyset ,$ it follows from the construction of $K$ that $d_{0}= dist\left ( \bar{K},S\cup P \right ) > 0 .$
	
	If there exists a sufficiently large $L= m-l$ such that some orbit point is still $L$ periods away from the periodic orbit. Take $\varepsilon = \frac{d_{0} }{2} ,$ then $$dist\left ( f^{k}\left ( x \right ),P   \right ) < \varepsilon = \frac{d_{0} }{2} .$$
	Supposing $f^{k}\left ( x \right )\in K,$ we obtain $$dist\left ( f^{k}\left ( x \right ),P \right ) \ge dist\left ( \overline{K},P  \right ) = d_{0} ,$$ which yields a contradiction.
	
	Therefore, if there exists $k> 0$ such that $f^{k}\left ( x \right )\in K  ,$ then it cannot be too far from eventually leaving the domain. Moreover, since the orbit of any preperiodic point in $S$ cannot enter $K$ before reaching the periodic orbit, and
	\begin{align}rm+n\left ( p \right ) -N_{1} < k< rm+n\left ( p \right ),\end{align}
	we conclude that the orbit of $x$ under $f$ can intersect $K$ only during its last $N_1$ iterates. Hence $rm+n\left ( p \right )< k+N_{1},$ which implies that $N_{1} $ is a uniform constant independent of $x$.
	
	This implies that the $f$-orbit between $x$ and $g\left ( x \right )$ hits $K$ at most $N_{1}  $ times. The orbit  intersecting with $K$ is only when it is outside of $U_{v} .$

	Now let $ K_{2}=\Lambda (G)\setminus  \bigcup_{p\in S\cup P}^{}V_{p}$ and $\widehat J$ denote the accesses space induced by the Markov partition corresponding to $G$. For each access $p\in S\cup P$, choose a one-sided cylinder neighborhood $V_{p} $ of $p$ and denote its closure in $\widehat J$ by $\overline{V}_p^{\,\widehat J}.$
	
	Set $V^{\,\widehat J}=\bigcup_{a\in S\cup P}\overline{V}_a^{\,\widehat J};
	$ $
	W^{\,\widehat J}=\widehat J\setminus V,$
	and  $\overline W^{\,\widehat J}
	=\widehat J\setminus \operatorname{Int}_{\widehat J}V .$
	Suppose that $p\in P$ is a parabolic periodic point and $p_{a} $ is one of its accesses. Since $V_{p_{a} } $ is a one-sided repelling petal based at $p_{a} ,$ its corresponding access endpoint is $p_{a}, $ hence $p_{a} \notin \overline W^{\,\widehat J}.$ All accesses of  rank-one parabolic periodic points (one point has two accesses) and  rank-two parabolic periodic points (one point has four accesses) have been removed. Therefore, no remaining access direction allows points in $W$ to approach $p.$ Equivalently, in the access topology, $\overline W^{\,\widehat J}\cap P=\varnothing .$
	
	The same argument applies to every one-sided access associated with a preperiodic point on the Markov boundary. Indeed, there is no other undeleted access adjacent to that point.
	
	Consequently, no point of $\overline W^{\,\widehat J}$ is a periodic point or a one-sided preperiodic boundary access. In particular, $\overline W^{\,\widehat J}\cap \left ( P\cup S \right ) = \emptyset .$
	
	Fix $x\in \overline W^{\,\widehat J}$  and define
	\begin{align}F_N (x):=\max_{1\le n\le N }\left | (f^n)'(x) \right | _{\sigma } .
	\end{align}
	
	Case 1: If the point $x$ is in $ J_{\mathrm{rad}},$ by the Koebe distortion theorem for inverse branches, $$\limsup_{n\to\infty}\left | (f^n)'(x) \right | _{\sigma } =\infty.$$ Hence there exists \(n(x)\ge 1\) such that $$\left | (f^n)'(x) \right | _{\sigma }>1.$$
	
	Case 2: If the point $x$ is is a nonperiodic preperiodic point.
	Let $n=n(p)\ge1$ be the minimal integer such that $f^n(x)=p$ for some $p\in P.$
	Since $x\notin\partial \mathfrak{R},$ it lies in the interior of an $n$-cylinder $C.$ Note that $$f^n|_C:C\longrightarrow f^n(C)$$
	is a diffeomorphism, let $\phi=(f^n|_C)^{-1}.$
	Then $\phi(p)=x.$
	By the strict contraction property of nontrivial inverse branches, $\left | \phi^{'} (p) \right |_{\sigma }  <1. $
	We obtain
	$$\left | (f^n)^{'} (x) \right |_{\sigma }
	=\frac1{\left | \phi^{'} (p) \right |_{\sigma }  }
	>1.$$
	
	In conclusion, for every $x\in \overline W^{\,\widehat J},$ there exists $n(x)\ge 1$ such that $$\left | (f^{n(x)})'(x) \right |_{\sigma }  >1.$$ If $\overline W^{\,\widehat J}$ also contains the outer endpoints of $V^{\,\widehat J},$ then either they belong to $ J_{rad }$ or they are among the  pre-parabolic points described above; hence they are  included in these two classifications.
	
	Fix $x\in \overline W^{\,\widehat J},$ and choose $n\left ( x \right ) $ as above. Let
	$\varepsilon_x
	:=
	\frac{\left | (f^{n(x)})'(x) \right |_{\sigma }  -1}{2}>0.$
	The map defined by $y\longmapsto \left| (f^{n(x)})'(y) \right |_{\sigma }$
	is continuous near $x$ in $\widehat J, $ since $\overline W^{\,\widehat J}$ contains no Markov-boundary points.  Hence there exists a relatively open neighborhood $V_x\subset\widehat J$ such that
	$$\left | (f^{n(x)})'(x) \right |_{\sigma } >1+\varepsilon_x,
	\qquad y\in V_x\cap \overline W^{\,\widehat J}.$$
	The family $\{V_x:x\in \overline W^{\,\widehat J}\}$ covers the compact set $\overline W^{\,\widehat J}.$ Choose a finite subcover $ V_{x_1}, V_{x_1}, \cdots,V_{x_s}$ of $\overline W^{\,\widehat J}.$
	
	Set $N_{2} =\max_{1\le i\le s}n_{x_i},
	$ $
	\delta=\min_{1\le i\le s}\varepsilon_{x_i}>0.$
	For any $y\in \overline W^{\,\widehat J},$ choose $i$ such that $y\in V_{x_i}.$ Then, $$\max_{1\le n\le N_{2}}\left | (f^n)'(y) \right |_{ \sigma }
	\ge
	\left | (f^{n\left ( x_i \right ) })'(y) \right | _{\sigma }
	>
	1+\varepsilon_{x_i}
	\ge 1+\delta.$$
	Thus $$\inf_{y\in \overline W^{\,\widehat J}}\max_{1\le n\le N_{2} }
	\left | (f^n)'(y) \right |_{\sigma }
	\ge 1+\delta>1.$$
	This implies that the $f$-orbit between $x$ and $g\left ( x \right )$ hits $K$ at most $N_{2}  $ times.

	In conclusion, for any $x\notin P\cup S ,$  let $N=\max \{N_{1}, N_{2}\}$.
	Then $\left | g^{'} \left ( x \right )  \right |_{\sigma }\ge \eta = \min \left \{ C_{1},C_{2},\lambda    \right \} > 1  $ and the $f$-orbit between $x$ and $g\left ( x \right ) $ hits $K$ at most $N$ times.
\end{proof}

%
%
%
%
%

%

\noindent{\bf Remark:}	Let $\mathfrak{R}= \left \{ R_{i}  \right \} _{i=0}^{\infty }  $ be a Markov partition of $\widehat{\mathbb{C}}$ induced by a geometrically finite Kleinian group.  Strictly speaking, the set on which the first return map is defined should initially be constructed as $$K
=
\Lambda(G)\setminus
\left(
\mathcal{P}\cup\bigcup_{s\in S\cup P}U_{s}
\right),$$
where $S\cup P$ is the countable set of parabolic points lying on the boundaries of the Markov partition in $\hat{ \mathbb{C}} $ associated with the Markov map $f,$ and $P\subset S\cup P$ is the finite set of parabolic vertices of the
chosen even-cornered fundamental domain.

The local structure of the Markov partition $\mathfrak{R}$ near a parabolic vertex implies that every neighborhood of each $p\in P$ contains countably many points of $S.$ Moreover, for each $p\in P,$   choose  a neighborhood $U_p$ of
the points $p\in P$  such that only finitely many points of $S$ are on the outside of $\bigcup_{p\in P}U_p.$ Hence, we can  make the neighborhoods of these finitely many exceptional points sufficiently small, after absorbing these finitely many exceptional neighborhoods into the notation for the sets
$U_p.$  Without loss of generality, we assume that $S\subset \bigcup_{p\in P}U_p.$
Accordingly, throughout the sequel we use the simplified notation
$$K
=
\Lambda(G)\setminus
\left(
\mathcal{P}\cup\bigcup_{p\in P}U_p
\right).$$

\section{Quasiconformal deformations of Kleinian groups}
Let $G$ be a geometrically finite Kleinian groups with parabolic elements.
For each $\mu\in \mathrm {M} \left ( G \right )$, note that when $g^{\mu}$ is  compatible with the group $G,$ the deformation group $G_{\mu}=g^{\mu}\circ G \circ (g^{\mu})^{-1}$ is also a subgroup of PSL$(2, \mathbb{C})$.
Fix $\mu\in \mathrm {M} \left ( G \right )$. By the self-similarity of the limit set of the group $G_{\mu},$ in order to study the dimension of the limit set of the group $G_{\mu},$ it is enough to study the dimensions of the image of $K$ under the quasiconformal map $g^{\mu }.$

Let  $\mathfrak{R}=\{R_{i}\}_{i\in I\cup \{0\}} $ be a Markov partition of $\hat{\mathbb{C}}$ for $G.$
Let
$$\bigtriangleup = \Lambda \left ( G \right )\bigcap \left(\bigcup_{i\in I\cup \left \{ 0 \right \} } \partial _{\hat{\mathbb{C} } }R_{i} \right ) .$$ From the assumption, it is known that $\bigtriangleup$ contains at most countably many elements. Moreover, since $R_{i}$ are pairwise mutually disjoint and $Int_{\hat{\mathbb{C} } } R_{0} \cap \Lambda \left ( G \right )= \emptyset.$
It is easy to see that $\Sigma = \bigcup_{g\in G}^{} g\left ( \bigtriangleup  \right ) $ contains at most countably many elements and is invariant under the action of $G.$ Note that every $p\in P $ lies on the boundary of some $R_{j}\in \mathfrak{R} .$ Therefore $\mathcal{P} \subset \Sigma .$ Let $\Lambda _{reg}=  \Lambda \left ( G \right ) \setminus \Sigma .$ Since both $\Lambda \left ( G \right )$ and $\Sigma$ are $G$-invariant, $\Lambda _{reg}$ is also $G$-invariant even though $\Lambda _{reg}$ is not closed.

A finite sequence $\omega = \left ( i_{0},...,i_{N}   \right ) $ is called  admissible if $f\left( R_{i_{k}} \right) \supseteq R_{i_{k+1}}$ for \\ $0\le k\le N-1,$ so that $$f^{N}\left( R_{i_{0}} \right) \supseteq f^{N-1}\left( R_{i_{1}} \right) \supseteq \dots \supseteq R_{i_{N}}.$$ An admissible word $\omega$ induces a truncated cylinder set $$R_{\omega } = \bigcap_{k= 0}^{N}f^{-k}\left ( Int_{\Lambda } R_{i_{k} }  \right )= \left \{ x\in \Lambda _{reg}:f^{k} \left ( x \right ) \in Int_{\Lambda } R_{i_{k} } ,0\le k\le N  \right \}    $$ of length $N+1,$ where $ Int_{\Lambda } R_{i_{k} } $
denotes $\Lambda(G)\bigcap R_{i_{k}}.$
All admissible cylinders of length $N+1$ form the $N$-th refinement of the original Markov partition. For each $N\ge 0,$ denote the resulting $N$-th refinement by $\mathfrak{R} ^{\left ( N \right ) } ,$ with $\mathfrak{R} ^{\left ( 0 \right ) }=  \mathfrak{R}.$ Since $\mathfrak{R}$ is at most countable, each $\mathfrak{R} ^{\left ( N \right ) } $ is also at most countable. By \cite{Mc2} we can see that, for every $N\ge 0,$ the higher-order refinement $\mathfrak{R} ^{\left ( N \right ) } $ remains a  Markov partition (at most countable) for $f.$
It should be noted that in order to ensure that the boundary of the set $K$ is compatible with the original Markov partition $\mathfrak{R}, $ we need to impose the additional requirement: there must exists a finite constant $N> 0$ such that $\partial K\subset \partial \mathfrak{R}^{(N)}.$ For every parabolic fixed point $p\in P$ and every integer $N\ge 0,$  define the star-shaped neighborhood of $p$  as $$U_{\mathrm{star}}^{N}(p)=\bigcup_{C\in \mathfrak{R}^{(N)},\,p\in C} C,$$
so that the neighborhood satisfies $p\in \operatorname{Int}_{\Lambda}U_{\mathrm{star}}^{N}(p)$ and
$$\operatorname{diam} U_{\mathrm{star}}^{N}(p)\to 0\qquad (N\to\infty).$$
Note that the set $ P$ of inequivalent parabolic fixed points of geometrically finite Kleinian groups is finite, the following choice of $N$ be made uniformly.
 Let
$$N\ge \max_{p\in P}N(p).$$
Define
$$U_{p}=\operatorname{Int}_{\Lambda}U_{\mathrm{star}}^{N}(p).$$
Then we have $p\in U_{p},$ $\operatorname{diam}U_{p}<\varepsilon_p$ and the boundary of this star neighborhood is composed of boundaries of $N$-th order Markov cylinders, so
$\partial_{\Lambda}U_{p}\subset \Sigma.$
Set
$$U_{\mathrm{par}}=\bigcup_{j=1}^{m}U_{p_j},\qquad
K=\Lambda(G)\setminus\left(U_{\mathrm{par}}\cup \mathcal{P}^{\ast}\right).$$
By a simple discussion, we can also choose  \(N\)  to satisfy the conditions of Theorem 3.1 and Theorem 3.8 before.

Let $K_{reg} = K\setminus \Sigma = K\cap \Lambda _{reg} .$ We know that, for points in $K_{reg},$ the orbit never lands on the boundaries of the selected cylinder star domain. Hence, the branch itinerary is well defined and unambiguous.

Let $\Omega_{N} $ be the set of all admissible words of length $N+1$ and let $\mathfrak{R} ^{\left ( N \right ) }$ be the $N$-th order refinement induced by $\Omega $ with $$\mathfrak{R} ^{\left ( N \right ) }= \left \{ R_{\omega }:\omega \in \Omega_{N}   \right \} .$$

We define the index set of parabolic states by $$\Omega_{U} = \left \{ \omega \in \Omega_{N} :R_{\omega }\in U_{p} , \:\:for\:\:some\:\:p\in P  \right \} ,$$ and set $$\Omega_{K} = \Omega_{N}\setminus \Omega_{U} ,$$ where $U_{p}$ is a neighborhood of $p.$
Since the boundary of $U_{par} $ is compatible with the Markov partition, $$U_{par} \setminus \Sigma = \bigcup_{\omega \in \Omega _{U} }^{} \left ( Int_{\Lambda }R_{\omega } \setminus \Sigma   \right ) $$ and $$K_{reg}= \bigcup_{\omega \in \Omega _{K} }^{} \left ( Int_{\Lambda }R_{\omega } \setminus \Sigma   \right ) .$$

We known that almost all elements of the countable partitions are concentrated in neighborhoods of the parabolic periodic points $P.$ Therefore, the index set $\Omega _{K} $ should be finite.

Let $f $ be the Markov map for $G$ with respect to the partition $\mathfrak{R}.$
Then there exists a subset $K$ of the limit set $\Lambda(G)$ such that the first return map $f_{K} :K_{reg} \to K_{reg}  $ satisfies the Markov properties.
More precisely, there exists an  Markov partition $\mathcal{A} = \left \{ A_{\alpha }  \right \} _{\alpha \in \Pi } $ of $K_{reg} $ such that:

1. For all $\alpha ,\beta \in \Pi ,$ $A_{\alpha } \cap A_{\beta }= \emptyset  ,$ and $K_{reg} = \bigcup_{\alpha \in \Pi }^{} A_{\alpha } .$

2. For each $\alpha \in \Pi ,$ there exists $\gamma _{\alpha } \in G$ such that $f_{K}\mid _{A_{\alpha } } =  \gamma _{\alpha } \mid _{A_{\alpha } } .$

3. The map $f_{K}$ satisfies the following transitivity condition: for every $A_{\alpha }\in \mathcal{A}  ,$ the image $f_{K} \left ( A_{\alpha } \right ) $ is a union of elements of $\mathcal{A}.$

From the construction of the Markov partition $\mathcal{A} = \left \{ A_{\alpha }  \right \} _{\alpha \in \Pi },$ we know that, for every $\alpha \in \Pi,$ there exists a unique admissible first return word decomposition as $\alpha = \left ( \omega _{0} ,\omega _{1}  ,...,\omega _{m} \right ),$ where $\omega _{0} ,\omega _{m}\in \Omega _{K},$ and $\omega _{i}\in \Omega _{U}, 0<  i<  m  .$ We prove that $f_{K} \left ( A_{\alpha }  \right )= B_{\omega_{m} }  .$

Given $[\mu] \in \mathcal{T}(G),$ let $f^{\mu }= g^{\mu }\circ f\circ\left ( g^{\mu } \right ) ^{-1}    $ be the corresponding Markov map for the deformation group $G_{\mu}$ with respect to the partition $\mathfrak{R}^{\mu}= g^{\mu } \left ( \mathfrak{R} \right ) ,$ then $f_{K}^{\mu }= g^{\mu }\circ f_{K} \circ\left ( g^{\mu } \right ) ^{-1}  $ is a Markov map on $K^{\mu }= g^{\mu } \left ( K \right )  $ with respect to the partition $\mathcal{A}^{\mu } =g^{\mu }  \left ( \mathcal{A}  \right ) .$
By Theorem 3.8, we know than the map $f_{K}$ is (eventually) expanding. In the following, we will shown that, under quasiconformal conjugation-deformations, the map $f_{K}^{\mu }  $ is also (eventually) expanding.

\begin{theorem}The map $f_{K}^{\mu }  $ defined as above is (eventually) expanding.
	
\end{theorem}
The method we used here is an expanding mirrors of the proof of Lemma 3 in  \cite{Bo1} or Proposition 2.3.2 in \cite{R2}.

\begin{proof}
	We can assume, without loss of generality, that $\infty \in \Omega \left ( G \right ) .$
	
	Let  $\mathcal{A} = \left \{ A_{\alpha }  \right \} _{\alpha \in \Pi } $ be the Markov partition $\mathcal{A} = \left \{ A_{\alpha }  \right \} _{\alpha \in \Pi } $ of $K_{reg} $ defined before the theorem.
	
	Suppose the sequence $\alpha = \left ( \omega _{0} ,\omega _{1}  ,...,\omega _{m} \right )$ is admissible, where $\omega _{0} ,\omega _{m}\in \Omega _{K},$ and $\omega _{i}\in \Omega _{U}, 0<  i<  m  .$ By the construction of the Markov partition of $f_K$ on $K,$  we know that $$f_{K} \left ( A_{\alpha }  \right )= B_{\omega _{m} }  \subset K,$$
	where $B_{\omega }=  R_{\omega }\cap \Lambda _{reg}  \: \: \: for\: \: all\: \: \omega \in \Omega .$
	
	Consider an $n$-cylinder of the Markov partition $\left \{ A_{\alpha _{j_{i} } }  \right \}_{\alpha _{j_{i} } \in \Pi }  $ in $K,$ corresponding to an admissible word $j= \left ( \alpha _{j_{0} },\alpha _{j_{1} },...,\alpha _{j_{n-1} } \right ) $ with $$E_{j}= A_{\alpha _{j_{0} }  }  \cap f_{K}^{-1}\left ( A_{\alpha _{j_{1} }  }   \right )\cap ...\cap f_{K} ^{-\left ( n-1 \right ) } \left ( A_{\alpha _{j_{n-1} } }   \right )   .$$
	Since $j$ is admissible, we have $f_{K}^{n} \left ( E_{j}  \right )= f_{K} \left ( A_{\alpha _{j_{n-1} } }  \right ) ,$ and hence $f_{K}^{n} \left ( E_{j}  \right )= B_{t_{\left ( j \right ) } } .$ Here $B_{t_{\left ( j \right ) } }$ is the last symbol of an admissible first-return word to $A_{\alpha _{j_{n-1} } } .$
	
	Note that $f_{K}^
	{n}\mid _{E_{j} }  $ is univalent. Assume  $$T_{j,n}= \left ( f_{K}^{n}\mid _{E_{j} }   \right )^{-1}:B_{t_{\left ( j \right ) } } \to E_{j}    .$$ Since $T_{j,n}\in G,$ the map $T_{j,n}$ extends to a globally defined biholomorphic automorphism of the Riemann sphere. For any $\omega  \in \Omega _{K} ,$ fix a complex neighborhood $D_{\omega  } $ of $\omega $ satisfying $\overline{B_{\omega } }\subset \operatorname{Int}(D_{\omega  } ).$   We extend the inverse branch $T_{j,n}$ associated with every cylinder of order $n$ terminating at $B_{\omega  }$  univalent to the  domain $D_{\omega  } .$
	
	For partition $\mathcal{A},$ let $\mathcal{A} ^{\left ( n \right ) }$ be the $n$-th order refinement induced by $\Omega .$
	Then for any $E_{j} \in \mathcal{A} ^{\left ( n \right ) } ,$ denote $$g^{\mu}(E_{j})=E_{j}^{\mu},\:\:\text{and}\:\:g^{\mu } \left ( T_{j,n}\left ( B_{t_{\left ( j \right ) } }  \right ) \right )=T_{j,n}^{\mu }\left ( B_{t_{\left ( j \right ) } }  \right ).$$
	Since $\overline{B}_{t_{\left ( j \right ) } } ^{\mu } \subset \operatorname{Int}(D_{t_{\left ( j \right ) }  }^{\mu }  ),$ we may choose a piecewise smooth Jordan curve $\ell  _{t_{\left ( j \right ) } } ^{\mu  } $ such that $$B_{t_{\left ( j \right ) }  }^{\mu } \subset Int\left ( \ell _{t_{\left ( j \right ) } } ^{\mu  }  \right ) ~~\text{and} ~~\overline{Int\left ( \ell _{t_{\left ( j \right ) } } ^{\mu  }  \right ) } \subset Int\left ( D_{t_{\left ( j \right ) }  }^{\mu }  \right ) .$$
	
	Let $$d_{t_{(j)}  }^{\mu } = dist\left ( B_{t_{\left ( j \right ) }  }^{\mu } ,\ell _{t_{\left ( j \right ) } } ^{\mu  }  \right )> 0 \:\:\: and \:\:\: L_{t_{(j)}  }^{\mu } = length\left ( \ell _{t_{\left ( j \right ) } } ^{\mu  }  \right ) < \infty .$$ For $T_{j,n  }^{\mu }:B_{t_{(j)}  }^{\mu } \to E_{j }^{\mu },$ the curve $\ell _{t_{\left ( j \right ) } } ^{\mu  } $ depends only on the terminal symbol.
	If two distinct $n$-cylinders $E_{j},E_{k}  $ have the same terminal symbol, namely $t\left ( j \right ) = t\left ( k \right ) ,$ we use exactly the same curve $\ell _{t_{\left ( j \right ) } } ^{\mu  } = \ell _{t_{\left ( k \right ) } } ^{\mu  } .$ Let  $\Pi ^{\left ( n \right ) }$ be the index set of the $n$-th cylinders set $\mathcal{A}^{n}.$ Define $$L_{\mu }= \sup _{n\ge 1,j\in \Pi ^{\left ( n \right ) } } \frac{length\left ( \ell _{t_{\left ( j \right ) } } ^{\mu  }  \right )}{2\pi dist\left ( B_{t_{\left ( j \right ) }  }^{\mu } ,\ell _{t_{\left ( j \right ) } } ^{\mu  }  \right )^{2 }}.$$
	Since $t\left ( j \right ) \in \Omega _{K} $ and $\Omega _{K}$ is finite, it follows that $$L_{\mu }= \sup_{n\ge 1,j\in \Pi ^{\left ( n \right ) } } \frac{L_{t_{\left ( j \right ) }  }^{\mu }}{2\pi \left ( d_{t_{(j)}  }^{\mu } \right )^{2}  }\le \max_{\omega  \in \Omega _{K} }\frac{L_{\omega   }^{\mu }}{2\pi \left ( d_{\omega   }^{\mu } \right )^{2}  } < \infty   .$$
	
	Fix an arbitrary $n$-cylinder $E_j.$ We know that $T_{j,n}^{\mu } :D_{t_{\left ( j \right ) } } ^{\mu } \to \hat{\mathbb{C} } $ is a univalent holomorphic map. For any $y\in B_{t_{\left ( j \right ) } } ^{\mu },$ the Cauchy integral formula gives $$\left ( T_{j,n}^{\mu }  \right ) ^{'} \left ( y \right ) = \frac{1}{2\pi i}\int_{\ell _{t_{\left ( j \right ) } } ^{\mu }}^{}\frac{T_{j,n}^{\mu } \left ( \zeta  \right ) }{\left ( \zeta -y \right )^{2}  } \mathrm{d}\zeta .$$ The residue theorem implies $$ \int_{\ell _{t_{\left ( j \right ) } } ^{\mu }}^{}\frac{1}{\left ( \zeta - y \right )^{2}  } \mathrm{d}\zeta=0.$$
	Hence  for any $y\in B_{t_{\left ( j \right ) } } ^{\mu },$  $$\left ( T_{j,n}^{\mu }  \right ) ^{'} \left ( y \right ) = \frac{1}{2\pi i}\int_{\ell _{t_{\left ( j \right ) } } ^{\mu }}^{}\frac{T_{j,n}^{\mu } \left ( \zeta  \right )-T_{j,n}^{\mu } \left ( y \right ) }{\left ( \zeta -y \right )^{2}  } \mathrm{d}\zeta.$$
	Furthermore, for any $\zeta \in \ell _{t_{\left ( j \right ) } } ^{\mu }$ and $y\in B_{t_{\left ( j \right ) } } ^{\mu },$  we have $$\left | T_{j,n}^{\mu } \left ( \zeta  \right )-T_{j,n}^{\mu } \left ( y \right ) \right | \le \text{diam } \left(T_{j,n}^{\mu } \left ( D_{t_{\left ( j \right ) } } ^{\mu } \right )\right )\:\:\:\text {and}\:\:\:\left | \zeta -y \right |\ge d _{t_{\left ( j \right ) } } ^{\mu }.
	$$
	Thus
	\begin{align*}\left | \left ( T_{j,n}^{\mu }  \right ) ^{'} \left ( y \right ) \right |  &\le  \frac{1}{2\pi }\int_{\ell _{t_{\left ( j \right ) } } ^{\mu }}^{}\frac{\left | T_{j,n}^{\mu } \left ( \zeta  \right )-T_{j,n}^{\mu } \left ( y \right )  \right | }{\left |  \zeta -y  \right | ^{2}  } \left | \mathrm{d}\zeta  \right |\\
		&\le \frac{L_{t_{\left ( j \right ) } } ^{\mu }}{2\pi \left ( d_{t_{\left ( j \right ) } } ^{\mu } \right )^{2}  }\text{diam}\left(T_{j,n}^{\mu } \left ( D_{t_{\left ( j \right ) } } ^{\mu } \right )\right).
		\\
		&\le L_{\mu }\text{ diam}\left(T_{j,n}^{\mu } \left ( D_{t_{\left ( j \right ) } } ^{\mu } \right )\right).
	\end{align*}
	For the map $T_{j,n}^{\mu }$ associated with any $n$-cylinder, we have $$\sup _{y\in B_{t_{\left ( j \right ) } }^{\mu } } \left | \left ( T_{j,n}^{\mu }  \right ) ^{'} \left ( y \right ) \right |  \le   L_{\mu } diamT_{j,n}^{\mu } \left ( D_{t_{\left ( j \right ) } } ^{\mu } \right )
	.$$
	Since $g^{\mu } $ is a quasiconformal homeomorphism and $D_{t_{\left ( j \right ) } }^{\mu } = g^{\mu } \left ( D_{t_{\left ( j \right ) } } \right ).$ Denote $$\Delta_n:=
	\sup_{\substack{(\gamma _0,\ldots,\gamma _{m-1})\ \mathrm{admissible}\\ m\ge n}}
	\operatorname{diam} I(\gamma _0,\ldots,\gamma _{m-1}),$$ by Rocha's Theorem 1.2.3(d) in \cite{R}, we know $\Delta_n\to0$ as $n\to\infty$. Since $\mathcal R$ refines the
	finite partition $\{I(\gamma )\}$, for every admissible word
	$\mathbf v=(v_0,\ldots,v_{m-1}),$ there are labels
	$\gamma (v_0),\ldots,\gamma (v_{m-1})$ such that $$C_{\mathbf v}\subset
	I(\gamma (v_0),\ldots,\gamma (v_{m-1})).$$
	Consequently, $$\sup_{\substack{\mathbf v\ \mathrm{admissible}\\|\mathbf v|\ge n}}
	\operatorname{diam}C_{\mathbf v}
	\le \Delta_n\longrightarrow 0\:\: \text{as}\:\: n\rightarrow \infty.$$

	Furthermore, suppose that the partition of $K\subset\Lambda$ is
	obtained by refining $\mathcal R$ according to the return times and
	the original Markov itineraries.  Each $n$-cylinder $C^K_{j}$ is then contained in an $\mathcal R$-cylinder whose
	length (at least $n$) is the sum of the first $n$ return times.
	Therefore
	$$\sup_{|j|=n}\operatorname{diam}C^K_{j}
	\le \Delta_n\to 0
	\qquad\text{as }n\to\infty.$$
	Rocha's result implies that $$\sup _{j\in \Pi ^{\left ( n \right ) } } \text{diam}\left(T_{j,n}\left ( D_{t_{\left ( j \right ) } }  \right )\right)\to 0 \:\:\text{as}\:\:n\to \infty. $$
	Consequently, $$\sup _{j\in \Pi ^{\left ( n \right ) } }\sup _{y\in B_{t_{\left ( j \right ) } }^{\mu } } \left | \left ( T_{j,n}^{\mu }  \right ) ^{'} \left ( y \right ) \right |  \le   A_{\mu }\sup _{j\in \Pi ^{\left ( n \right ) } } diamT_{j,n}^{\mu } \left ( D_{t_{\left ( j \right ) } } ^{\mu } \right )\to 0 $$ as $ n\to \infty  .$
	
	We can choose sufficiently large $n$  and  choose $0< q< 1$ so that, for every $n$-cylinder $E_j,$ $$\sup _{y\in B_{t_{\left ( j \right ) } }^{\mu } } \left | \left ( T_{j,n} ^{\mu }  \right )^{'} \left ( y \right )   \right |\le q< 1 ,$$ and for any $E_{j} \in \mathcal{A} ^{\left ( n \right ) } ,$ we define $T_{j,n}^{\mu }\left ( B_{t_{\left ( j \right ) } }  \right )= g^{\mu }  \left ( T_{j,n}\left ( B_{t_{\left ( j \right ) } }  \right ) \right ) .$ Since all $n$-cylinders cover $K^{\mu } ,$ it follows that $$\inf _{x\in K^{\mu } } \left | \left ( \left ( f_{K}^{\mu }  \right )^{n}   \right ) ^{'}\left ( x \right )   \right |\ge \frac{1}{q}> 1  .$$ Therefore $f^{\mu }_{K}$ is  an expanding map.

\end{proof}

It is well known that whenever $\{\mu_{t}\}$ is an analytic family of Beltrami coefficients in $\mathrm {M} (G)$ with respect to the parameter $t$, the family of quasiconformal maps $\{g^{\mu_{t}}\}$ is also an analytic family with respect to $t.$ Hence, $G_{\mu_{t}}$ is an analytic family of quasiconformal deformations of $G.$ Recall that a discrete group $G$ is called satisfying Ruelle's property if whenever $G_{\mu_{t}}$ is an analytic family of quasiconformal deformation groups of $G,$ then dim$\Lambda (G_{\mu_{t}})$ is an analytic function of $t,$ where dim$(\Lambda)$ denotes the Hausdorff dimension its limit set. By Theorem 3.8 and Theorem 4.1, we can show the main theorem of the paper.

\begin{theorem}\label{main}
	Let $G$ be a geometrically finite Kleinian group which has an even-cornered fundamental polyhedron.  Then $G$ has Ruelle's property.
\end{theorem}

\begin{proof}
	Suppose $\mu$ is a $G$-compatible Beltrami coefficient. Let $K$ be a subset of the limit set  $\Lambda(G)$ constructed in Theorem 3.1.
	Then the map $f_{K} $ satisfies the Markov properties on $K.$ Hence we can get an iterated system(IFS) $S$ by $f_{K}$.  Conjugating  $S$ by $g^{\mu } ,$ we get an IFS $S_{\mu}$ as
	$$S_{\mu}= \{ \phi_\alpha  ^{\mu} : \phi_\alpha  ^{\mu}= g^{\mu }  \circ \phi_\alpha  \circ (g^{\mu } ) ^{-1},\; \alpha  \in \Pi \}.$$
	Here the countable index set $\Pi$  is induced by the Markov partition $\mathcal{A} =\left \{ A_{\alpha } \right \} $ on $K$ such that for any $A_{j } \in \mathcal{A},$  $f_{K} ^{-1} \mid _{A_{\alpha  } } = \phi _{\alpha  } \mid _{A_{\alpha  } }.$   This establishes a one-to-one correspondence between the Markov partition $\mathcal{A}$ and the countable index set $\Pi.$
	
	Let $\Pi_{n}^{\mu}$  be the  index set of the set of all quasiconformal deformations of the $n$-th cylinder partitions $\mathcal{A}^{(n)}. $ We have
	$$\psi_n^{\mu}(\iota ) = \sum_{j\in \Pi_n^{\mu}} \| (\phi_{j}^{\mu})' \|^{\iota },\qquad n\ge 1$$
	and let $P$ be the topological pressure function
	$$P(\mu,\iota ) = \lim_{n\to\infty} \frac{1}{n} \log \psi_n^{\mu}(\iota ),\qquad \iota   \in (\theta_{\mu},+\infty),$$
	where $\theta_{\mu} = \inf_{\iota } \{ \iota  : \psi_1^{\mu}(\iota ) < \infty \}.$
	
	If the set of generators of $G_{\mu}$ contains no parabolic elements, then the index set $\Pi$ is finite. Thus $\theta_{\mu}= -\infty$ and the system is regular. Anderson \cite{AR} proved this result for convex co-compact Kleinian groups with an even-cornered fundamental polyhedron. Their method utilizes the existence of a unique zero of the Ruelle zeta function in the complex plane $\mathbb{C},$ which can be represented by the pressure function $P(\mu,\iota ).$ When the set of generators of $G_{\mu}$ contains some parabolic elements, we need the following lemma, for the case of Fuchsian groups see \cite{HZ}.
	
	\begin{lemma} For any $[\mu]\in \mathcal{T}(G),$ we have $\theta_{\mu} = \frac{k_{max}}{2},$ where $k_{max}$ denotes the largest cusp rank. Furthermore the IFS $S_{\mu}$ is regular.
	\end{lemma}
	
	\begin{proof}
		
		Fix $[\mu]\in\mathcal T(G)$. By definition,
		$$\theta_\mu=
		\inf\left\{
		\iota :
		\psi_1^\mu(\iota )<\infty
		\right\},$$
		where
		$$\psi_1^\mu(\iota )
		=
		\sum_{\alpha\in \Pi}
		\left\|(\phi_\alpha^\mu)'\right\|^\iota .$$
		It suffices to prove that
		$$\psi_1^\mu(\iota)<\infty
		\:\:\text{for}\:\:
		\iota>\frac{k_{\max}}{2},$$
		and
		$$\psi_1^\mu\left(\frac{k_{\max}}{2}\right)=\infty.$$
		
		Let $ P_\mu$ be a finite set of representatives of the $G_\mu$-orbits of bounded parabolic fixed points of $G_\mu$. For each $p\in P_\mu$, let $k(p)\in\{1,2\}$ denote the rank of the corresponding cusp, and choose a finite-index subgroup
		$L_p<  \operatorname{Stab}_{G_\mu}(p)$ such that $L_p\cong\mathbb Z^{k(p)}.$
		Set
		$k_{\max}
		=
		\max_{p\in P_\mu}k(p).$
		We decompose the index set $\Pi$ as
		$$\Pi
		=
		\Pi_h
		\cup
		\bigcup_{p\in P_\mu}\Pi_p,$$
		where $\ \Pi_h$ corresponds to the hyperbolic branches and $ \Pi_p$ corresponds to the parabolic branches associated with the cusp $p$. Consequently,
		$$\psi_1^\mu(\iota)
		=
		\sum_{\alpha\in \Pi_h}
		\left\|(\phi_\alpha^\mu)'\right\|^\iota
		+
		\sum_{p\in P_\mu}
		\sum_{\alpha\in \Pi_p}
		\left\|(\phi_\alpha^\mu)'\right\|^\iota.$$
		By the construction of the Markov map $f$ and the definition of the induced map $f_K$, the set $\Pi_h$ is finite \cite{BS}. Hence, for every $\iota\geq0$,
		$$\sum_{\alpha\in\Pi_h}
		\left\|(\phi_\alpha^\mu)'\right\|^\iota
		<\infty.$$

		For the parabolic case, fix  $p\in\mathcal P_\mu$ and let
		$r=k(p).$
		After  changing  coordinates by Mobius conjugation-transforation, we may assume that $p=\infty$.  Choose $L_p$  to act as a lattice of translations on
		$$\mathbb R^r\subset\mathbb R^2.$$
		The orbit map
		$$\gamma \longmapsto \gamma (0),
		\qquad \gamma \in L_p$$
		identifies $L_p$ with a discrete lattice in $\mathbb R^r$. 	More generally, the cusp stabilizer may also have a finite rotational part acting in the directions orthogonal to $\mathbb R^r$, but this does not affect the convergence of the series considered below.
		By the bounded distortion property of the induced Markov system and the construction of the parabolic branches, with comparison constants independent of $\gamma \in L_p$, we have
		$$\sum_{\alpha\in\Pi_p}
		\left\|(\phi_\alpha^\mu)'\right\|^\iota
		\asymp
		\sum_{\gamma \in L_p}
		|\gamma '(0)|_\sigma^\iota,$$
		where $|\gamma '(0)|_\sigma$ denotes the derivative norm with respect to the spherical metric
		$$ds_\sigma=\frac{2|dz|}{1+|z|^2}.$$
		The spherical derivative of $\gamma $ is given by
		$$|\gamma '(z)|_\sigma
		=
		|\gamma '(z)|
		\frac{1+|z|^2}{1+|\gamma (z)|^2}.$$
		We have $|\gamma '(0)|=1,$ since the elements of $L_p$ act as Euclidean isometry in the above coordinates. Therefore
		$$|\gamma '(0)|_\sigma
		=
		\frac{1}{1+|\gamma (0)|^2}.$$
		Hence
		$$\sum_{\alpha\in\Pi_p}
		\left\|(\phi_\alpha^\mu)'\right\|^\iota
		\asymp
		\sum_{\gamma \in L_p}
		(1+|\gamma (0)|^2)^{-\iota}.$$
		
		Since $\{\gamma (0):\gamma \in L_p\}$ is a full-rank lattice in $\mathbb R^r$, the preceding lattice sum has the same convergence behavior as
		$$\int_{\mathbb R^r}
		(1+|x|^2)^{-\iota}\,dx.$$
		Using polar coordinates, we obtain
		$$\int_{\mathbb R^r}
		(1+|x|^2)^{-\iota}\,dx
		=
		\operatorname{area}(S^{r-1})
		\int_0^\infty
		\frac{t^{r-1}}{(1+t^2)^\iota}\,dt.$$
		As $t\to\infty$, the radial integrand satisfies
		$$\frac{t^{r-1}}{(1+t^2)^\iota}
		\asymp
		t^{r-1-2\iota}.$$
		Thus, the integral, and consequently the corresponding lattice sum, converges if and only if
		$$r-1-2\iota<-1$$
		or equivalently,
		$$\iota>\frac r2.$$
		
		When $\iota=r/2$, the integrand is asymptotic to $t^{-1}$, and hence both the integral and the corresponding lattice sum diverge. Therefore, for every $p\in P_\mu$,
		$$\sum_{\alpha\in\Pi_p}
		\left\|(\phi_\alpha^\mu)'\right\|^\iota
		<\infty
		\quad\Longleftrightarrow\quad
		\iota>\frac{k(p)}{2}$$
		and
		$$\sum_{\alpha\in\Pi_p}
		\left\|(\phi_\alpha^\mu)'\right\|^{k(p)/2}
		=
		\infty.$$
		Since $ P_\mu$ is finite and the contribution from the hyperbolic branches is always finite, it follows that
		$$\psi_1^\mu(\iota)<\infty
		\quad\Longleftrightarrow\quad
		\iota>
		\frac12\max_{p\in P_\mu}k(p)
		=
		\frac{k_{\max}}2.$$
		In particular, if $p$ is any cusp of rank $k_{\max}$, then the parabolic contribution diverges at $\iota=k_{\max}/2$. Hence
		$$\psi_1^\mu\left(\frac{k_{\max}}2\right)=\infty.$$
		It follows that
		$$\theta_\mu=\frac{k_{\max}}2.$$
		Finally, by the regularity criterion for countable conformal iterated function systems cited above,
		$$S_\mu\text{ is regular}
		\quad\Longleftrightarrow\quad
		\psi_1^\mu(\theta_\mu)=\infty.$$
		Thus the system $S_\mu$ is regular.

	\end{proof}

	Mauldin and Urbanski \cite{MU} showed that, for a regular system, the dimension of the limit set is the unique zero of the function $\iota \mapsto P(\mu,\iota).$ To finish the proof of the theorem, it remains only to prove that the zero varies real analytically with respect to $\left [ \mu  \right ] .$ This follows from the classical thermodynamic formalism (a generalization of the Perron–Frobenius theorem, see \cite{Bo} \cite{Ru}): $\exp P(\mu,\iota)$ is an isolated eigenvalue of a transfer operator. The theorem follows from the implicit function theorem applied to $(\mu,\iota) \mapsto \exp(P(\mu,\iota)).$
\end{proof}

\section{Hausdorff dimensions of limit sets of Hecke groups}

In \cite{Mc3}, McMullen  gave an eigenvalue algorithm for accurately computing the Hausdorff dimension of the limit sets of some classes of convex cocompact groups Kleinian groups. McMullen's algorithm can be applied to Schottky groups and Hecke groups. In this chapter, our objective is to extend McMullen's results regarding convex cocompact groups to the context of geometrically finite groups. We will also provide calculations for relevant examples.

The $\sigma$-algebra of $\mu$-measurable sets is the $\sigma$-algebra of Borel sets of $\mathbb{B} ^{3} .$ For $\delta \ge 0,$ a measure $\mu$ is called a $\delta$-conformal measure of $G$ if  $\mu$ is a non-trivial finite measure supported on $\Lambda(G)$  satisfying the transformation rule
$$\mu(\gamma A) = \int_A |\gamma '|^{\delta} \, d\mu$$
for every Borel set $A$ of $\mathbb{B} ^{3} $ and every $\gamma \in G,$ where $|\gamma '|$ is the operator norm of the derivative of $g.$ Here the derivative given by $\sigma = 2|dx|/(1+|x|^2)$ on $\mathbb{S} ^{2} $ is measured in the spherical metric.




\subsection{Schottky groups} Every circle $C\subset \hat{\mathbb{C}}$ bounds a plane $H(C)\subset \mathbb{H}^{3}$, which is a Euclidean hemisphere normal to $\hat{\mathbb{C}}.$  The space of all circles forms a manifold, naturally identified with the Grassmannian space $\mathcal{H}^{3}$ of hyperbolic planes in $\mathbb{H}^{3}.$

Consider a configuration $\mathcal{C}= \{ C_1, C_2, \dots, C_d \}$ of $d\ge 3$ circles in $\hat{\mathbb{C}}$ bounding disks $D(C_i)\subset S_{\infty}^{2}$ with disjoint interiors:
$$\operatorname{int}(D(C_i)) \cap \operatorname{int}(D(C_j)) = \emptyset \qquad (i\ne j).$$
These circles can be tangent but they cannot overlap or nest. The space of such configurations forms a closed set $\mathcal{S}_d\subset (\mathcal{H}^{3})^{d}.$
Let $\operatorname{int}\mathcal{S}_d$ be the set of configurations which contains all the pairwise disjoint circles. Let $\partial \mathcal{S}_d$ denote the set of configurations which contains at least one pair of  circles tangent with each other.

Any circle $C\subset \hat{\mathbb{C}}$ determines a conformal reflection $\rho : S_{\infty}^{2} \to S_{\infty}^{2}$ fixing $C.$  For any $\mathcal{C}\in \mathcal{S}^{d}$, let $\Gamma(\mathcal{C})$ be the Schottky group generated by the reflections $\rho_i$ through $C_i,$ $i=1,\dots,d.$

Each plane $H(C_i)$ divides the space $\mathbb{H}^3$ into two half spaces. The region which is outside all the solid hemispheres bounded by $\langle H(C_i) \rangle$ for $C_{i} \in\mathcal{C}= \{ C_1, C_2, \dots, C_d \}$ gives a finite sided fundamental domain for $\Gamma(\mathcal{C}).$ Consider the exterior intersection of these planes
$$\mathfrak{F} = \bigcap_{i=1}^{d} \bigl( \mathbb{H}^{3}\setminus \operatorname{int}(H_i^{+}) \bigr),$$
where $H_i^{+}$ is the solid hemisphere (i.e. the side containing the disk $D(C_i)$ bounded by $H(C_i)$). More concretely $\mathfrak{F}$ is the common region exterior to all these hemispheres.
Since the interiors of all these disks are pairwise disjoint, the projections of all these hemispheres onto the boundary $\hat{\mathbb{C}}$ are mutually non-overlapping. Hence $\mathfrak{F}$ is a polyhedron with finitely many faces (in fact a polyhedron whose faces are parts of hemispheres, together with possibly some ideal parts coming from the boundary at infinity). Thus the Schottky group $\Gamma(\mathcal{C})$ is geometrically finite.

The orientate double-cover of $M = \mathbb{H}^{3}/\Gamma(\mathcal{C})$ is a handlebody of genus $d-1.$ Any cusp of $M$ is rank one  corresponding bijectively to a pair of tangent circles in $\mathcal{C}.$ In this case, we have $\mathcal{C}\in \partial \mathcal{S}_d.$ Thus $\Gamma(\mathcal{C})$ is convex co-compact if and only if the circles in $\mathcal{C}$ are pairwise disjoint ($\mathcal{C}\in \operatorname{int}\mathcal{S}_d$).

In this paper,  by extending McMullen's Theorem 3.1 \cite{Mc3} to  the case of geometrically finite groups with parabolic elements,  we will  show that the analytic property of $\!\dim \Lambda(\Gamma(\mathcal{C}))$ not only holds inside $\mathcal{S}_d$ , but also persists on the boundary of $\mathcal{S}_d.$

\subsection{ Hecke groups}

For \(0<r\le 1\), let $C_r=\{z\in\mathbb H:|z|=r\}$ be the semicircle in the upper half-plane and $$\ell_{+}=\{z\in\mathbb H:\operatorname{Re}z=1\}\:\: \text{and}
\:\:
\ell_{-}=\{z\in\mathbb H:\operatorname{Re}z=-1\}$$ be  two vertical geodesics with  \(z=\infty\) as the common endpoint point . Let $\mathcal C_r=\langle C_r,\ell_{+},\ell_{-}\rangle$ denote the geometric configuration formed by the three geodesics and  $\Gamma_r=\Gamma(\mathcal C_r)$ be the Fuchsian reflection group generated by the reflections in these geodesics, which is also called Hecke group. Let $\Gamma _{n} \subset \mathrm {Isom} \left ( \mathbb{H}^{2}   \right )   $ be a sequence of Kleinian groups. There are several possible notions of convergence of the sequence $\Gamma _{n}.$ We say $\Gamma _{n} \to \Gamma _{A} $ $algebraically$ if there exist isomorphisms $\chi _{n} :\Gamma _{A} \to \Gamma _{n} $ such that $\chi _{n} \left ( g \right )\to g $ for each $g\in \Gamma _{A} .$ A parabolic element $g\in \Gamma $ is an $accidental$ $parabolic$ if $\chi _{n} \left ( g \right ) $ is hyperbolic for infinitely many $n.$ Then the complex length $L_{n} +i\theta _{n} $ of $\chi _{n} \left ( g \right ) $ tends to zero; it does so $radially$ if $\theta _{n}=O\left ( L_{n}  \right )  ,$ and $horocyclically$ if $\theta _{n}^{2} /L_{n}\to 0 .$ It is clear that the sequence of Hecke groups we construct for $\Gamma _{r} \longrightarrow \Gamma _{1} $ is algebraically convergent.  Denote the limit set of $\Gamma_r$ by $\Lambda_r.$ For the other type of limit, We say $\Gamma _{n} \to  \Gamma _{G} $ $geometrically$ if we have convergence in the Hausdorff topology on closed subsets of $ \mathrm {Isom} \left ( \mathbb{H}^{2}   \right ) .$ If $\Gamma _{n} $ converges to $\Gamma$ both geometrically and algebraically, then it converges strongly. By studying the moduli space of surfaces, we can show that, for the other type of limit, all rank-one accidental parabolic points correspond to radial convergence. Since a Fuchsian group contains only rank-one parabolic points, it follows that, in the sequence $\Gamma _{r}\to \Gamma _{1}  $ we construct, all accidental parabolic limit points of the geometrically finite group $\Gamma _{1}$ are radially. For a detailed analysis, see  \cite{Mc1}.

To obtain the deformation on the quotient surface, we must also prove that $\Gamma _{r}\to \Gamma _{1}  $ is geometrically convergent. Applying Theorem 7.3 of McMullen \cite{Mc1}, we find that the sequence $\Gamma _{r}\to \Gamma _{1}  $ is strongly convergent and hence also geometrically convergent.
Let $X_r$ denote the orientate double cover of the quotient surface $\mathbb H/\Gamma_r.$ It is easy to see that for $0<r<1,$ $X_r$ is a hyperbolic pair of pants with one rank-one cusp and two ends of infinite volume. The cusp corresponds to the point $z=\infty$ in the fundamental domain.

When \(r=1\), the circle $C_1=\{z:|z|=1\}$ is tangent to both $$\ell_{+}=\{\operatorname{Re}z=1\},
\:\: \text{and} \:\:
\ell_{-}=\{\operatorname{Re}z=-1\}.$$ Consequently, the limiting group $\Gamma_1$ of $\Gamma_r$ as $r\rightarrow1$ has two rank-one parabolic elements arising from the two tangent pairs. Its limit set degenerates to the extended real line $\Lambda_1=\widehat{\mathbb R}$ and hence $\!\dim_{H}(\Lambda_1)=1.$

Let $\rho_r(z)=\displaystyle\frac{r^2}{\overline z}$ be the reflection regarding to the circle $C_r$. It is easy to see that $\rho_r$ is an anti-conformal map. In what follows, let $|D\rho_r(z)|$ denote the norm of its differential. So that $$|D\rho_r(z)|
=
\frac{r^2}{|z|^2}.$$

Let us consider the family of translations $\{T_{2n}(z)=z+2n\}.$
This family is closely related to Hecke groups.
Let $H_R$ be the group generated by
$$\left\langle
z\mapsto-\frac1z,\ z\mapsto z+R
\right\rangle$$ and $G_R$ be the group generated by
$$\left\langle
z\mapsto\frac{z}{2Rz+1},\ z\mapsto z+1
\right\rangle,$$ where $R=\frac2r.$
According to McMullen's result \cite{Mc2}, the groups $\Gamma_r,$ $H_R,$ and $G_R$ have a common finite cover. More precisely, there exists a hyperbolic surface $Y_r$ such that all of the three maps $$Y_r\longrightarrow\mathbb H/\Gamma_r,
\qquad
Y_r\longrightarrow\mathbb H/H_R,
\qquad
Y_r\longrightarrow\mathbb H/G_R$$ are  finite covering maps. Hence, these three groups have the same critical exponent and their limit sets have the same Hausdorff dimensions.

In this section, we regard $\{\Gamma_r\}_{0<r<1}$ as a real-analytic family of groups respect to the parameter $r.$ Here “real-analytic” means that, with respect to fixed generators and fixed coordinates, the generating reflections and the representing matrices of every group element are real-analytic in $r.$

We must point out   that here $\Gamma_r$ does not be a quasiconformal conjugating deformation of a fixed group. Although one can construct a quasiconformal map sending the unit circle to the circle $|z|=r,$ this does not  guarantee that it conjugates every reflection in the original fundamental domain to the corresponding reflection in the new fundamental domain. Therefore, the real-analyticity in $r$ below will be obtained directly from the group generators, the countable Markov partition, and the Ruelle's transfer operator.
Now we give the main result of this section. Let $\delta(r)$ denote the \poin convergence exponent of the \poin series of the geometrically finite Kleinian group $\Gamma_{r}.$  It is well known that  $\delta(r)$ is equal to Hausdorff dimension of the limit set $\Lambda(\Gamma_r)$ of the group $\Gamma_{r}.$ About the function  $\delta(r)$, we have

\begin{theorem}
Let $\delta(r)$ denote the \poin convergence exponent of the Hecke group $\Gamma_{r}.$  The function $\delta(r)$ is real-analytic with respect to $r$ in the interval $(0,1).$
Moreover, $\delta(r)$ can be extended continuously to the closed interval $[0,1].$ Under the self-similar decomposition of the critical conformal measure, $\delta(r)$ satisfies the normalization equation
\begin{align}
1=
2\sum_{n=1}^{\infty}
\left(\frac r{2n}\right)^{2\delta(r)}
a_n(\delta(r)),
\end{align}
where $a_n(\delta(r))$ is the correction factor arising from the distortion of the branches, and $a_n(\delta(r))\longrightarrow 1
$ as $n\to\infty.$ Furthermore, $\delta(r)
=
\frac{1+r}{2}+O(r^2)$ as $r\to 0,$  and if $X_r$ denotes the orientable double cover of $\mathbb H/\Gamma_r,$ then $$1-\delta(r)
\sim
\lambda_0(X_r)
\asymp
\sqrt{1-r} \:\: \text{as}\:\: r\to1.$$

\end{theorem}

\begin{proof}

The proof will be divided into three parts.

First, we  consider the case $\delta(r) \longrightarrow \frac{1}{2}$ as $r\to 0$  to ensure the continuity of $\delta(r)$ at $0.$ The method we adopted  here is from McMullen (\cite{Mc2}, Theorem 3.6 ).  It can be see that for sufficiently small $r$ , the limit set $\Lambda_r$ is a Cantor set.
After removing the infinity point $\infty$  and its countable orbit under the group action, the limit set can be decomposed into a sequence of connected                                                   branches . Let $\Lambda_r(0)$ denote the basic branch containing the origin, and let $\Lambda_r(n)$ denote the branch concentrated near $2n,$ where $n\in\mathbb Z.$ Then up to a countable set with measure zero, there is a self-similar decomposition $$\Lambda_r(0)
=
\rho_r\left(
\bigcup_{n\ne0}\Lambda_r(n)
\right),$$ where $\rho_r$ is the reflection about the circle $C_r.$

Let $\mu_r$ be a $\delta(r)$ dimension critical conformal measure of $\Gamma_r.$ Since the derivatives of the translations $T_{2n}$ are modulus $1$ in the Euclidean metric,  all the translated branches have the same measure $\mu_r(\Lambda_r(n))
=
\mu_r(\Lambda_r(0)).$
By conformal invariant of the measures $\mu_{r}$, we obtain $$\mu_r(\Lambda_r(0))
=
\sum_{n\ne0}
\int_{\Lambda_r(n)}
|D\rho_r(z)|^{\delta(r)}
\,d\mu_r(z).$$
For sufficiently large $|n|$ and $z\in\Lambda_r(n)$ , the branch $\Lambda_r(n)$ is concentrated near $2n.$ Hence  $|z|\asymp2|n|.$  Furthermore notice that $|D\rho_r(z)|=\displaystyle \frac{r^{2}}{|z|^{2}}$, we have $$|D\rho_r(z)|^\delta
=
\left(\frac r{2|n|}\right)^{2\delta}
\bigl(1+\text{distortion term}\bigr).$$
By the symmetry between $n$ and $-n,$ we obtain the critical normalization equation $$1=
2\sum_{n=1}^{\infty}
\left(\frac r{2n}\right)^{2\delta(r)}
a_n(\delta(r)),$$ where $a_n(\delta(r))\longrightarrow1
,$ as $n\to\infty.$

If we temporarily ignore the branch distortion and set $a_n(\delta,r)\sim 1,$ we obtain the model equation
\begin{align} 1\sim
2\sum_{n=1}^{\infty}
\left(\frac r{2n}\right)^{2\delta}.
\end{align}

Let us consider the Riemann zeta function $$\zeta(s)
=
\sum_{n=1}^{\infty}\frac1{n^s},
\qquad
\operatorname{Re}s>1.$$ Then (5.2) becomes $1
\sim
2\left(\frac r2\right)^{2\delta}
\zeta(2\delta),$ or equivalently
\begin{align}2\zeta(2\delta)
\sim
\left(\frac r2\right)^{-2\delta}.
\end{align}

Now let  $r\to 0,$ note that  the Riemann zeta function has a unique pole at $s=1,$ it follows that $2\delta(r)\longrightarrow1, $ i.e. $\delta(r)\longrightarrow\frac{1}{2} $ as $r\to 0.$

Note that (5.2) is a asymptotic estimation formula, in the following  we will give a more precise estimation to prove that $\delta(r) = \frac{1+r}{2} + O(r^2) \quad \text{as } r \to 0,$ which is a result of McMullen.

Set $\delta=\frac12+\varepsilon,
$ $ \varepsilon>0.$ For $s$ near $1$, the Riemann zeta function has the Laurent expansion
$$\zeta(s)
=
\frac1{s-1}
+\mathcal{C }_{0}
+O(s-1),$$ where $\mathcal{C }_{0} $ is the Euler-Mascheroni constant defined as
$$\mathcal{C }_{0}
=
\lim_{N\to\infty}
\left(
\sum_{n=1}^{N}\frac1n-\log N
\right).$$
Thus \begin{align}\zeta(1+2\varepsilon)
=
\frac1{2\varepsilon}
+\mathcal{C }_{0}
+O(\varepsilon).
\end{align}
Substituting (5.4) into the model equation (5.3), we can have

\begin{align}\frac1{\varepsilon}
+2\mathcal{C }_{0}
+O(\varepsilon)
=
\frac2r
\exp\left(
2\varepsilon\log\frac2r
\right)= \frac{2}{r}\Bigl(1 - 2\varepsilon\ln(r/2) + O(\varepsilon)\Bigr).
\end{align}
Let $\varepsilon = cr + o(r),$ it follows from (5.5) that
$$\frac{1}{cr} + 2\mathcal{C }_{0}  + O(r) = \frac{2}{r}\Bigl(1 - 2cr\ln(r/2) + O(r)\Bigr).$$
Multiplying both sides by $r$, we have
$$\frac{1}{c} + 2\mathcal{C }_{0}  r + O(r) = 2 - 4cr\ln(r/2) + O(r).$$
Note that $r\log\frac1r\longrightarrow0$ as $r\to 0,$ the dominant terms in (5.5) imply $\varepsilon\asymp r,$ which implies $c = 1/2,$ and thus $\varepsilon = \frac{r}{2} + O(r).$
Furthermore, this gives $$\delta(r)
=\frac12+\frac r2+O(r).$$

If we reconsider the error produced by the factor $\exp\left(
2\varepsilon\log\frac2r
\right),$ then the model equation gives $$\delta(r)
=
\frac{1+r}{2}
+
O\left(r^2\log\frac1r\right).$$
However, the complete self-similar equation also contains the  distortion factors $a_n(\delta(r)).$ Therefore the rough model  (5.2) alone does not directly imply an error term $O(r^2)$.
A more precise estimate of the leading eigenvalue of the complete transfer operator is given by McMullen’s eigenvalue method (Theorem 3.6 \cite{Mc2}) which implies  $$\delta(r)
=\dim_{H}(\Lambda_r)
=
\frac{1+r}{2}+O(r^2),
\qquad r\to0.$$
Consequently, after defining $\delta(0)=\frac12,$ the function $\delta(r)$ is continuous at $r=0.$

Second,  we will prove  that $\delta(r) \to 1 \quad \text{as } r \to 1.$

For $r=1,$ the circle $C_1$ is tangent to  $\ell_{+}$ and $\ell_{-}.$ Hence the limit set is $\Lambda_1=\widehat{\mathbb R}$ and $\delta(1)
=
\dim_{H}(\Lambda_1)
=
1.$

Let $r\to 1.$ Since the circles $C_{r}$ converge to $C_1,$ the reflections $\rho_{r}$ in these circles converge to $\rho_1.$  Note that the reflections about the two vertical lines   $\ell_{+}$ and $\ell_{-}$ are  independent of $r.$ Hence the generators converge, and therefore $\Gamma_{r}\longrightarrow\Gamma_1$ algebraically, for the definition of algebraical convergence see \cite{Mc1}.


For $r<1,$ the circle $C_r$ is disjoint from the line $\ell_{+}.$ Hence the product of the reflections in these two geodesics is a hyperbolic element. Let $d_r=d_{\mathbb H^2}(C_r,\ell_{+}),$ denote the hyperbolic distance between the two geodesics. The translation length of the product of the two reflections is $2d_r.$ Consequently, its multiplier is $\lambda_r=e^{2d_r},$ or its reciprocal, depending on the chosen of fixed points.

As $r\to1,$ the circle $C_r$ becomes tangent to $\ell_{+},$ and therefore $d_r\longrightarrow0.$ It follows that $\lambda_r\longrightarrow1.$
Moreover, $\lambda_r$ approaches $1$ along the positive real axis, which is precisely the condition for radial convergence of the accidental parabolic element.

By symmetry, the same argument applies to the geodesic pair $(C_r,\ell_{-}).$ Hence he limiting group arise from the geodesic pairs $(C_1,\ell_{+}), (C_1,\ell_{-})$ contains two accidental parabolic elements.

In \cite{Mc2}, McMullen showed that
if $\Gamma_n\to\Gamma$ algebraically and  all accidental parabolics converge radially, then $\Gamma_n\to\Gamma$ dynamically.
In the present family, all accidental parabolic elements arise from the two tangent pairs described above, and we have verified that they converge radially. Therefore, $\Gamma_{r}\longrightarrow\Gamma_1$ dynamically, which implies $\dim_{H}(\Lambda_{r_j})
\longrightarrow
\dim_{H}(\Lambda_1) \:\:\text{as}\:\: r_{j} \to 1.$
We conclude that $\delta(r)\longrightarrow1,
$ as $ r\to1.$ Thus, after defining $\delta(1)=1,$ the function $\delta(r)$ is continuous at $r=1.$

In order to prove the analytic property, we need a more precise quantitative estimate as $r\to1.$ In the following, we will provide a quantitative estimate for $1 - \delta(r) \sim \lambda_0(X_r) \asymp \sqrt{1 - r} $ as $ \to 1.$
To do so,  we divide the quotient surface into three parts, separating the convex core from the ends of infinite volume.

Let $X_r$ denote the orientate double cover of the quotient surface $\mathbb H/\Gamma_r.$ Consider the convex core of the surface $X_r.$ For $0<r<1,$ it is easy to see that $X_r$ has a geometrically finite hyperbolic pair of pants with one rank-one cusp and two ends of infinite volume. Its convex core, denoted by $\operatorname{Core}(X_r)$,  has finite area. The boundary of $\operatorname{Core}(X_r)$ consists of two simple closed geodesics of equal length, denoted their common length by $L_r.$ Furthermore, the two  geodesics divide $X_r$ into three parts:
$$X_r=
\operatorname{Core}(X_r)
\cup E_r^+
\cup E_r^-,$$ where $E_r^+$ and $E_r^-$ are the two ends of infinite volume.
As $r\to1,$ we can see   $L_r\longrightarrow0$  and the two boundary geodesics degenerate into the two rank-one cusps of the limiting surface.

Denote the bottom of the spectrum of the surface $X_{r}$ by $\lambda_0(X_r).$		The variational characterization of the bottom of the spectrum is
\begin{align}\lambda_0(X_r)
=
\inf_{0\ne f\in C_c^\infty(X_r)}
\frac{\displaystyle\int_{X_r}|\nabla f|^2\,dA}
{\displaystyle\int_{X_r}|f|^2\,dA}.
\end{align}
The bottom of the spectrum of the Laplacian is controlled by the short geodesics. In the following, we will use the standard short-geodesic estimate to give $\lambda_0(X_r)\asymp L_r.$

For the asymptotic estimates at $r\to 1,$ the method we used here is from \cite{Mc2} and \cite{Mc3}.

By McMullen \cite{Mc1} , we know that $\Gamma_r\to\Gamma_1$ converges strongly as $r\to 1.$ Hence the two infinite-volume ends degenerate into two rank-1 cusps. At this point, two accidental parabolic points are generated. These short geodesics control the least eigenvalue of the Laplacian. We have
\begin{align}\lambda_0(X_r)
=
\inf_{0\ne f\in C_c^\infty(X_r)}
\frac{\displaystyle\int_{X_r}|\nabla f|^2\,dA}
{\displaystyle\int_{X_r}|f|^2\,dA}.
\end{align}
since the quotient above is approximately minimized by a function with $f=1$ in $\operatorname{Core}(X_r)$, and $f=0$ in the infinite volume ends, and $\nabla f$ supported in standard collar neighborhoods of the short geodesics, see \cite{JT}.

In the following, we  explain that the areas of standard collar neighborhoods of the short geodesics are controlled by the length of the core curves.
Let $\eta  $ be a simple closed geodesic of length $l.$ The half-width $w(\eta )$ of the standard embedded collar satisfies $$\sinh w(\eta )\,\sinh\frac l2=1.$$
Using Fermi coordinates $(\rho,t),$ the hyperbolic metric in the collar is $d\rho^2+\cosh^2\rho\,dt^2,$ where $-w(\eta )\le\rho\le w(\eta  )$ for $
0\le t\le l.$
Hence the area of the entire collar is
$$\operatorname{Area}(\operatorname{Coll}(\eta  ))
=\int_0^l
\int_{-w(\eta )}^{w(\eta )}
\cosh\rho\,d\rho\,dt.$$
Computing the  integral gives $$\operatorname{Area}(\operatorname{Coll}(\eta ))
= 2l\sinh w(\eta ).$$
Note that $\sinh w(\eta )
=
\frac1{\sinh(l/2)},$ we obtain $$\operatorname{Area}(\operatorname{Coll}(\eta ))
=
\frac{2l}{\sinh(l/2)}.$$
Let $l$ tend $0,$ we have $\operatorname{Area}(\operatorname{Coll}(\eta ))
\longrightarrow 4.$
Thus the area of the collar remains bounded. This allows us to compare the numerator with $L_r.$ In the following we will give the estimation of the asymptotic behavior of the length $L_r$ as $r$ approaching 1.

By symmetry, it suffices to compute the distance between $C_r=\{z\in\mathbb H^2:|z|=r\}$ and the vertical geodesic $\ell_{+}
=\{z\in\mathbb H^2:\operatorname{Re}z=1\}.$

The two boundary endpoints of $C_r$ in $\mathbb{R}$ are $-r$ and $r,$ while the two boundary endpoints of $\ell_{+}$ are $1$ and $\infty.$ By the cross-ratio formula for the distance between two disjoint geodesics in the upper half-plane model, we have $$\cosh d_{\mathbb H^2}(C_r,\ell_{+})
=
\frac{1+r}{2r}.$$
Therefore, $$d_{\mathbb H^2}(C_r,\ell_{+})
=
\operatorname{arcosh}
\left(
\frac{1+r}{2r}
\right).$$
Note that the length of the closed geodesic corresponding to the product of the two reflections is twice the length of the common perpendicular. Hence
\begin{align}L_r
=
2d_{\mathbb H^2}(C_r,\ell_{+})
=
2\operatorname{arcosh}
\left(
\frac{1+r}{2r}
\right).
\end{align}
Observe that $\frac{1+r}{2r}
=1+\frac{1-r}{2r}.$ Set $x=\frac{1-r}{2r}.$ Then by some easy calculations we have
\begin{align}d_{\mathbb H^2}(C_r,\ell_{+})
\asymp
\sqrt{\frac{1-r}{r}}.
\end{align}
Combining (5.8) and (5.9), we have $$L_r
\asymp
\sqrt{1-r},
\qquad r\to1.$$


Since $\Gamma_r$ is a non-elementary geometrically finite Fuchsian group with  parabolic elements, its critical exponent satisfies $\delta(r)>\frac{1}{2}.$
By Sullivan's spectral formula and together with the related results of Bishop–Jones and Sullivan \cite{BJ}\cite{DS3}\cite{DS4}, we have $$\lambda_0(X_r)
=
\delta(r)\bigl(1-\delta(r)\bigr).$$

In general, for $M=\mathbb H^{d+1}/\Gamma,$ the bottom of the spectrum satisfies $$\lambda_0(M)
=
\begin{cases}
\dfrac{d^2}{4},
& \delta(\Gamma)\le\dfrac d2,\\[6pt]
\delta(\Gamma)\bigl(d-\delta(\Gamma)\bigr),
& \delta(\Gamma)\ge\dfrac d2.
\end{cases}$$
Hence for the case of Fuchsian groups ($d=1,$ )
we have $$\lambda_0(X_r)
=
\delta(r)\bigl(1-\delta(r)\bigr)
\sim
1-\delta(r).$$
Thus, $1-\delta(r)
\sim
\lambda_0(X_r).$ On the other hand, $\lambda_0(X_r)\asymp L_r$ and $L_r\asymp\sqrt{1-r}.$
%
Combining these estimates gives $$1-\delta(r)
\sim
\lambda_0(X_r)
\asymp
L_r
\asymp
\sqrt{1-r},
\qquad r\to1.$$
%
%

Now its time to give the proof of the real analyticity of the Hausdorff dimensions of the limit set $\Lambda(\Gamma_{r})$ for $ r \in (0, 1).$
We firstly construct the Ruelle transfer operator $(\mathcal L_{s,r}u)(x)$ for this group using the group elements.

Fix $r\in(0,1)$. After removing the parabolic fixed points in the limit set and their countable orbits under the group action, the basic branch, denoted by $\Lambda_r(0)$, admits a countable Markov partition. Denote by $\Sigma
=
(\mathbb Z\setminus\{0\})^{\mathbb N}$ the symbolic space
with the left shift map
$$\sigma(n_1,n_2,n_3,\ldots)
=
(n_2,n_3,\ldots).$$
For each $n\ne0$, define the inverse branch by
$$f_{n,r}
=
\rho_r\circ T_{2n},
\qquad
T_{2n}(z)=z+2n.$$
For a finite word
$$\omega=(n_1,\ldots,n_k)
\in(\mathbb Z\setminus\{0\})^k,$$
let
$$f_{\omega,r}
=
f_{n_1,r}\circ\cdots\circ f_{n_k,r}.$$
Note that
$|DT_{2n}|=1$
and
$|D\rho_r(z)|
=
\frac{r^2}{|z|^2},$
we have
$$|Df_{n,r}(z)|
=
\frac{r^2}{|z+2n|^2}$$
and
$$|Df_{\omega,r}(z)|
=
\prod_{j=1}^{k}
\left|
Df_{n_j,r}
\left(
f_{n_{j+1},r}\circ\cdots\circ f_{n_k,r}(z)
\right)
\right|.$$
Thus for sufficiently large $|n|$,
$$|Df_{n,r}(z)|
\asymp
\left(\frac r{2|n|}\right)^2.$$
Therefore, for every $s>\frac{1}{2}$,
$$\sum_{n\ne0}
\sup_z |Df_{n,r}(z)|^s
<\infty.$$
For $s>\frac12$, define the Ruelle transfer operator by
$$(\mathcal L_{s,r}u)(x)
=
\sum_{n\ne0}
|Df_{n,r}(x)|^s
u(f_{n,r}(x)).$$
%
Since this system has countably many branches, we use the transfer-operator theory for countable Markov systems rather than the ordinary matrix theory for finite-branch systems. Let $\mathcal L_{s,r}$ act on a suitable space of analytic functions on an appropriate complex neighborhood.
According to Theorem 1.2, the operator family
$$(s,r)\longmapsto\mathcal L_{s,r}$$
is real-analytic in $(s,r)$, and its maximal eigenvalue is a positive, isolated simple eigenvalue. Denote this maximal eigenvalue by $\lambda(s,r).$
Hence there exist a strictly positive eigen-function $h_{s,r}$ and a positive eigen-measure $\nu_{s,r}$ normalized by
$$\int h_{s,r}\,d\nu_{s,r}=1$$ such that
$$\mathcal L_{s,r}h_{s,r}
=
\lambda(s,r)h_{s,r}$$
and
$$\mathcal L_{s,r}^{*}\nu_{s,r}
=
\lambda(s,r)\nu_{s,r}.$$
Define the equilibrium measure
$$\mu_{s,r}
=
h_{s,r}\nu_{s,r}.$$
For a word of length $k$, define the partition function $Z_k(s,r;x)$ as
$$Z_k(s,r;x)
=
\sum_{|\omega|=k}
|Df_{\omega,r}(x)|^s.$$
By the iteration formula for the transfer operator, we have
$$\mathcal L_{s,r}^{\,k}\mathbf1(x)
=
Z_k(s,r;x).$$
According to the Ruelle-Perron-Frobenius theorem,
$$Z_k(s,r;x)
=
\lambda(s,r)^k
\bigl(h_{s,r}(x)+o(1)\bigr),$$
where the error estimate is understood in the sense of exponential growth.
Similar calculations can be referenced in Rocha's \cite{R}, Anderson and Rocha's \cite{AW}, and Ruelle's \cite{Ru}.
Therefore, the topological pressure is defined by
$$P(s,r)
=
\lim_{k\to\infty}
\frac1k\log Z_k(s,r;x),$$
and this limit is independent of the initial point $x$.

By Gelfand’s formula for the spectral radius, we have
$$\lambda(s,r)
=
\lim_{k\to\infty}
\|\mathcal L_{s,r}^{\,k}\|^{1/k},$$
where the operator norm is taken in the space of continuous functions. More generally, the spectral radius can be expressed as:
$$\lambda(\delta,r) = \lim_{k\to\infty} \left( \sum_{n_1,\dots,n_k\neq 0} \prod_{i=0}^{k} |f_{n_i}'(x_{i-1})|^{\delta} \right)^{1/k},$$
where $x_{i-1}=f_{n_1}\cdots f_{n_{i-1}}(x_0)$ and the limit is independent of the initial point $x_0.$
Furthermore,  there exists a positive eigenfunction $h$ such that $\mathcal{L}_{\delta,r}h = \lambda(\delta,r)h.$ For the special case of the Hecke groups, we can derive the expression for the spectral radius as
$$\lambda(\delta,r) = 2\sum_{n=1}^{\infty} \left(\frac{r}{2n}\right)^{2\delta} a_n(\delta,r).$$
and hence
$$P(s,r)=\log\lambda(s,r).$$

In \cite{MU},  Mauldin and Urbanski  showed that for a regular system, the dimension of the limit set is the unique zero of the function $r \mapsto P(\delta(r),r)$, i.e. the Hausdorff dimension is the unique zero of the pressure function with respect to $s.$ Hence we only need to take the unique zero point of the equation $\lambda(\delta(r), r) = 1.$

%
Applying the implicit function theorem to $P(\delta(r),r),$  we have  the existence of a unique zero point of the equation $\lambda(\delta(r), r) = 1.$

%

By Lemma 4.3 of this article, since the group $\Gamma _{r} $ contains parabolic elements, hence the regular index is $\theta =1/2.$ We can directly use Lemma 3.2 in \cite{MU} to conclude that $P(s,r)$ is strictly decreasing in the variable $s$ on the interval $\left ( 1/2,+  \infty  \right ) .$
Moreover, by Theorem 1.2 we can conclude that $P(\delta,r)$ is real-analytic in the region $(0,1)\times(1/2,1).$ The real-analytic implicit function theorem implies that there exists a neighborhood $U_{r_0}$ of $r_0$ and a unique real-analytic function
$$\widetilde\delta:U_{r_0}\longrightarrow\mathbb R$$
such that
$$P(\widetilde\delta(r),r)=0,
\qquad
\widetilde\delta(r_0)=s_0.$$

On the other hand, since $P(s,r)$ is strictly decreasing in $s$, the equation
$P(s,r)=0$
has at most one zero for each fixed $r$. Hence,
$\widetilde\delta(r)=\delta(r)$
on $U_{r_0}$. Since $r_0\in(0,1)$ was arbitrary, we conclude that
$r\longmapsto H.\!\dim(\Lambda_r)$
is real-analytic on $(0,1)$.
\end{proof}

From \cite{NP} and \cite{Su3}, it follows as a corollary of the Mostow rigidity theorem that the limit set of a geometrically finite Kleinian group supports a unique non-atomic $\Gamma$-invariant density $\mu$ of total mass one. The dimension of the canonical density $\mu$ agrees with the Hausdorff dimension of the limit set. 

\end{document}